\documentclass{article}

\usepackage[usenames, dvipsnames]{color}

\usepackage{epsfig,amsmath,amsthm}
\usepackage{graphicx}
\usepackage{amssymb}
\usepackage{enumerate}
\usepackage{comment}
\usepackage{color}
\usepackage{paralist}
\usepackage[left=2cm,right=2cm,top=2cm,bottom=2cm]{geometry}
\usepackage{fancyhdr}
\usepackage{hyperref}
\hypersetup{
	colorlinks=true,
	linkcolor=blue,
	citecolor=red,
	urlcolor=blue,
	pdfborder={0 0 0}
}

\usepackage{upgreek}
\usepackage{amsmath,amssymb,amsthm,amsfonts}
\usepackage{mathtools}
\usepackage{dsfont}
\usepackage{graphicx}
\usepackage{subcaption}
\usepackage{float}
\usepackage{bbm}
\usepackage{bm}
\usepackage{enumitem}
\usepackage{pifont}

\newcommand{\PP}{{\mathbb{P}}}
\newcommand{\TVS}{{\mathbb{A}}}
\newcommand{\EE}{{\mathbb{E}}}
\newcommand{\D}{{\mathbb{D}}}
\newcommand{\LL}{{\mathcal{L}}}
\newcommand{\eps}{{\epsilon}}

\newtheorem{theorem}{Theorem}[section]
\newtheorem{corollary}[theorem]{Corollary}
\newtheorem{lemma}[theorem]{Lemma}

\newtheorem{proposition}[theorem]{Proposition}

\theoremstyle{definition}

\theoremstyle{remark}

\theoremstyle{remark}

\title{Thick points of the planar GFF are totally disconnected for all $\gamma\ne 0$}
\author{Juhan Aru \thanks{\'{E}cole Polytechnique F\'{e}d\'{e}rale de Lausanne} \and L\'{e}onie Papon \thanks{Durham University}   
\and
Ellen Powell \footnotemark[2]}
\date{}

\begin{document}

\maketitle

\begin{abstract}
We prove that the set of $\gamma$-thick points of a planar Gaussian free field (GFF) with Dirichlet boundary conditions is a.s. totally disconnected for all $\gamma \neq 0$. Our proof relies on the coupling between a GFF and the nested $\mathrm{CLE}_4$. In particular, we show that the thick points of the GFF are the same as those of the weighted $\mathrm{CLE}_4$ nesting field {introduced in \cite{nesting_lim} and establish the almost sure total disconnectedness of the complement of a nested $\mathrm{CLE}_{\kappa}$, $\kappa \in (8/3,4]$. As a corollary we see that the set of singular points for supercritical LQG metrics is a.s. totally disconnected.}
\end{abstract}

\section{Introduction}

The two-dimensional continuum Gaussian free field (GFF) first appeared in the context of Euclidean quantum field theory to model the free massless bosons \cite{SimonEQFT}. From a mathematical perspective, the study of this object is motivated by its connections to many other planar models. For example, its rich interplay with Schramm-Loewner evolutions (SLE) and conformal loop ensembles (CLE) has led to a deeper understanding of these objects, and vice versa \cite{DubSLE,localset,IG_1}. 
Moreover, the two-dimensional GFF is conjectured, and in some special cases proved, to be the scaling limit of a wide class of lattice models at criticality, such as the height fluctuation of the dimer model \cite{dimers_lim}, 
or the Ginzburg-Landau model \cite{GLmodel}. It also arises in the context of random planar maps via Liouville quantum gravity \cite{cvg_triang} and plays a major role in the probabilistic construction of certain conformal field theories, e.g., \cite{LQG_sphere}. In fact, the GFF can be characterized as the only scale invariant planar field that satisfies a natural domain Markovian property, \cite{GFFchar}, which goes some way to explaining this ubiquity.

\subsection{Thick points of the GFF and main results}
In this work, we study the set of thick points of the GFF with Dirichlet boundary conditions (Dirichlet GFF). A GFF $h$ with Dirichlet boundary conditions in an open and simply connected domain $D \subset \mathbb{C}$ is a centered Gaussian process indexed by smooth functions that are compactly supported in $D$. Its covariance is given by, for $f, g \in \mathcal{C}_{c}^{\infty}(D)$,
\begin{equation*}
    \EE[(h,f)(h,g)] = \int_{D \times D} f(x)G_{D}(x,y)g(y) dx dy
\end{equation*}
where $G_D$ is the Green function of the Laplacian in $D$ with Dirichlet boundary conditions, normalised so that $\Delta G_D(x,\cdot) = -\delta_{x}(\cdot)$. As $G_{D}(x,x)=\infty$, the process $(h,f)_{f\in \mathcal{C}_c^\infty(D)}$ does not correspond to integration against a pointwise defined function. It does, however, almost surely correspond to an element of the Sobolev space $\mathcal{H}^{-\eps}(D)$, $\eps > 0$, i.e. a distribution, or generalised function. The conformal invariance of $G_D$ implies that $h$ is itself conformally invariant in the sense that if $\varphi: D \to \tilde D$ is a conformal map, then $\tilde h$ defined by {$(\tilde h, f) := (h, \vert \varphi ' \vert^{2} (f \circ \varphi^{-1}))$}, $f \in \mathcal{C}_{c}^{\infty}(\tilde D)$, is a Dirichlet GFF in $\tilde D$. See, for example, \cite{book_GFF} for proofs of these facts and further background.

For $h$ a Dirichlet GFF in $D$, the set of thick points of $h$ is a special set of points at which, loosely speaking, $h$ takes atypically high or low values. As $h$ is not defined pointwise, this set must be defined by regularisation. Let $z \in D$ and $r > 0$ and denote by $\rho_r^z$ the uniform measure on $\partial B(z,r)$ where $B(z,r)$ is the ball centered at $z$ of radius $r$. We consider the random variable $h_r(z):=(h,\rho_{r}^z)$ which is well-defined, e.g. by taking limits, 
since the integral $\int G_{D}(x,y) \rho_r^z(dx) \rho_r^z(dy)$ is finite. 

In fact, by \cite[Proposition~2.1]{thick_points}, if $h$ is a GFF with Dirichlet boundary conditions in the unit disc $\mathbb{D}:= \{ z \in \mathbb{C}: \vert z \vert <1 \}$, then $(h_r(z))_{r,z}$ has a version  
such that with probability one, for every $\alpha \in (0,1/2), \upzeta \in (0,1)$ and $\epsilon>0$ there exists a (random) constant $M=M(\alpha,\upzeta,\epsilon)<\infty$ such that for all $z,w \in \mathbb{D}$ and $s,r\in (0,1]$ with $1/2<r/s<2$,
\begin{equation}\label{eq_versionintro}
    \vert h_r(z) -  h_{s}(w) \vert \leq M \bigg( \log \frac{1}{r} \bigg)^{\upzeta} \frac{\vert (z,r) - (w,s)\vert^{\alpha}}{r^{\alpha+\epsilon}}.
\end{equation}
In the rest of the paper, we will only work with this version of the circle average process.

For fixed $z\in D$, a direct calculation shows that the process $h_{e^{-t}}(z)$ actually evolves as a linear Brownian motion in $t$. In particular, $\lim_{r \to 0} h_r(z)/\log(1/r)=0$ almost surely.
However, this does not rule out the existence of exceptional points at which this limit is non-zero: these points are called the thick points of $h$. It is natural to define, for $\gamma \in \mathbb{R} $, the set of $\gamma$-thick points of $h$ by
\begin{equation}\label{thick_def}
    \mathcal{T}_{\gamma}(h) := \{ z \in D: \lim_{r \to 0} \frac{h_r(z)}{\log 1/r} = \frac{\gamma}{\sqrt{2\pi}} \}
\end{equation}
where the factor $1/\sqrt{2\pi}$ comes from our choice of normalisation for the Green function. Note that since we work with a H\"{o}lder continuous version of the circle average process, as in \eqref{eq_versionintro}, there is an event of probability one on which we can determine the existence (or not) of the limit in \eqref{thick_def} for all $z$ in $D$ simultaneously. That is, the set $\mathcal{T}_\gamma(h)$ is well defined with probability one.

By \cite[Corollary~1.4]{thick_points}, the set $\mathcal{T}_\gamma(h)$ set is also conformally invariant in the following sense. If $\varphi: D \to \tilde D$ is a conformal map, then almost surely for any $\gamma \in [-2,2]$, $\varphi(\mathcal{T}_{\gamma}(h))=\mathcal{T}_{\gamma}(\tilde h)$ where $\tilde h$ (defined as the image of $h$ under $\varphi$ as above) is a GFF with Dirichlet boundary conditions in $\tilde D$.  Moreover, as shown in \cite{thick_points}, if $\vert \gamma \vert > 2$, this set is almost surely empty and if $\gamma \in [-2,2]$, it almost surely has Hausdorff dimension $2-\gamma^2/2$. {In particular, if $\gamma=0$, then $\mathcal{T}_{0}(h)$ almost surely has Hausdorff dimension 2: $0$-thick points are typical, as discussed above.}

Here, we prove another geometric property of $\mathcal{T}_{\gamma}(h)$. {Recall that a set $U$ is said to be totally disconnected if for each point $x \in U$, the connected component of $x$ in $U$ consists of just that point $x$. By \cite[Proposition~3.5]{Falconer}, a sufficient condition for a set to be totally disconnected is that this set has Hausdorff dimension strictly less than 1. In particular, observe that if $\vert \gamma \vert > \sqrt{2}$, then $\mathcal{T}_{\gamma}(h)$ has almost sure Hausdorff dimension strictly less than 1, which therefore implies that $\mathcal{T}_{\gamma}(h)$ is almost surely totally disconnected.} One may wonder whether this property extends to the full range $\gamma \in [-2,2] \setminus \{0\} $. The answer to this question is positive and this is the main result of our work. By conformal invariance, we may restrict ourselves to the case where $h$ is a Dirichlet GFF in $D=\mathbb{D}$ where $\mathbb{D}:=\{z \in \mathbb{C}: \vert z \vert < 1 \}$ is the complex unit disc.

\begin{theorem} \label{th_thickpts}
Let $h$ be a GFF with Dirichlet boundary conditions in $\mathbb{D}$. Then almost surely, $\mathcal{T}_{\gamma}(h)$ is totally disconnected for all $\gamma \in [-2,2]\setminus \{0\}$.
\end{theorem}

The proof of this result is based on a coupling of the Dirichlet GFF with a nested version of CLE$_4$. This coupling, and the construction of nested CLE$_\kappa$, will be described precisely below, but let us simply say for now that it gives rise to a \emph{different}, but natural, definition of the set of $\gamma$-thick points for the GFF, $\Phi_\gamma(h)$, defined via its so-called weighted CLE$_4$ nesting field, as studied in \cite{nesting_lim, nesting_extreme} and that we recall in equation \eqref{def_nesting4} below. It will follow rather immediately from the total disconnectedness of the complement of nested CLE$_4$ (the result of Theorem \ref{th_CLE} of this paper) that $\Phi_\gamma(h)$ is almost surely a totally disconnected set. Theorem \ref{th_thickpts} is then a consequence of the following result, that is of independent interest, and can be thought of as a universality statement for different notions of GFF-thick points. Results of this kind have also been obtained in \cite{Cipri} where it was shown that the thick points of the GFF defined via various approximations of the field -- convolution with a smooth mollifier, white noise approximations and integral cut-offs of the covariance -- agree, provided the approximations satisfy some regularity and second moment assumptions.

\begin{theorem}\label{th_intro_nestingfield}
Let $h$ be a GFF in $\mathbb{D}$ with Dirichlet boundary conditions. Then, with probability one, $\mathcal{T}_\gamma(h)\equiv \Phi_\gamma(h)$ for every $\gamma\in [-2,2]\setminus \{0\}$.
\end{theorem}

\subsection{Application to Liouville quantum gravity}{
The exponential of the GFF first appeared in the physics literature in the context of Euclidean quantum field theories \cite{QFT}, where it is used to describe the exponential interaction. It has also been used in the description of fluctuating strings, and in relation to 2D toy models of quantum gravity \cite{Polyakov}.} The exponential of the GFF now plays a major role in random conformal geometry, where it is used to give a probabilistic construction of Liouville Conformal field theory (LCFT), solve random welding problems and investigate the continuum limit of random planar maps, to mention only a few applications. One of the most important objects in these contexts is the Liouville quantum gravity (LQG) measure. It depends on a parameter $\gamma \in (0,2)$ and can informally be defined as $\mu_{\gamma}(dz) = e^{\gamma h(z)}dz$ where $h$ is a Dirichlet GFF. As $h$ is not defined pointwise, the rigorous construction of $\mu_{\gamma}$ involves a regularisation procedure. When performed appropriately, it has been shown that this yields a limiting (atomless) measure for every $\gamma\in (0,2)$; see \cite{LQG_measureB} for an elementary exposition. This measure is intimately connected to thick points of the underlying field. Indeed, if one samples a point according to the normalised LQG measure $\mu_{\gamma}$ with parameter $\gamma \in (0,2)$, then this point is almost surely a $\gamma$-thick point of the field used to construct $\mu_{\gamma}$. 

{Another object of interest in the context of Liouville quantum gravity (and thus  random planar maps) is the so-called LQG metric which can be thought of as a conformal perturbation of the Euclidean metric by the exponential of the GFF.} Although the construction of the LQG metric is more involved than that of $\mu_h$, it has now been succesfully carried out in a series of works, \cite{tightness_sub,uniqueness_sub,tightness_super,uniqueness_super}.

For a parameter $\xi > 0$, the LQG metric in a disk $\D$ is formally defined by
\begin{equation} \label{def_metric}
    D_{{h}}^{\xi}(z,w) = \inf_{P: z \to w} \int_{0}^{1} e^{\xi h(P(t))} \vert P'(t) \vert dt,
\end{equation}
where the infimum is over all continuous paths from $z$ to $w$ inside $\D$ and $h$ is a GFF. The definition (\ref{def_metric}) is purely formal as ${h}$ is not defined pointwise. To properly construct the LQG metric, one defines rescaled approximations of $D_{h}^{\xi}$ using a regularised version of ${h}$ and then shows tightness of these approximations in an appropriate topology. The final step is to prove that any subsequential limit must satisfy a natural set of axioms that uniquely characterises the metric. 

The properties of the LQG metric crucially depend on the parameter $\xi$ in \eqref{def_metric}. In particular, by \cite{tightness_super, singularpts}, there exists a unique $\xi_{crit} > 0$ such that if $\xi > \xi_{crit}$, then the metric with parameter $\xi$ almost surely does \emph{not} induce the Euclidean topology on {$\mathbb{D}$}. Instead, such a metric, called supercritical, admits a set of singular points: these points are at infinite distance from every other point. We denote by $S_{{h}}^{\xi}(\mathbb{D})$ this set of singular points of $D_{{h}}^{\xi}$, that is
\begin{equation*}
    S_{{h}}^{\xi}(\mathbb{D}) := \{ z \in \mathbb{D} : D_{{h}}^{\xi}(z,w) = \infty  \quad  \forall \, w \in \mathbb{D} \setminus \{z\} \}.
\end{equation*}
This set is intimately related to thick points of $h$. Indeed, \cite[Proposition~1.11]{singularpts} shows that there exists $Q(\xi) \in (0,2)$ such that
\begin{equation*}
    S_{{h}}^{\xi}(\mathbb{D}) = \{ z \in \mathbb{D}: \limsup_{r \to 0} \frac{{h}_r(z)}{\log 1/r} > Q(\xi) \} \quad \text{almost surely.}
\end{equation*}
{A consequence of the proof of Theorem \ref{th_thickpts} is the almost sure total disconnectedness of $S_{{h}}^{\xi}(\D)$.
}
\begin{proposition} \label{prop_singular}
{Let $\xi > \xi_{crit}$. Then $S_{{h}}^{\xi}(\D)$ is almost surely totally disconnected. } 
\end{proposition}

\subsection{Outline of the proof and further results}

Let us now discuss the proof of Theorem \ref{th_thickpts}. As mentioned above, it relies on a coupling of a nested CLE$_4$, $\Gamma$, as the set of so-called level lines of a GFF $h$ with Dirichlet boundary conditions. Non-nested CLE$_\kappa$ are a family of probability distributions on ensembles of non-nested loops (closed curves) in open and simply connected domains of the complex plane \cite{CLE_She, CLE_Markovian}. They are well-defined for $\kappa \in (8/3,8)$, and connected to the Schramm--Loewner evolution with parameter $\kappa$ via the so-called branching tree construction \cite{CLE_She}. 
The geometry of the loops 
depends on the value of $\kappa$: when $\kappa \in (8/3,4]$, these loops are almost surely simple loops that do not intersect each other or the boundary of $D$; on the contrary, when $\kappa \in (4,8)$, they are almost surely non-simple but non-self-crossing and they may touch (but not cross) the boundary of $D$ and each other. A \emph{nested} CLE$_\kappa$ is constructed from a non-nested CLE$_\kappa$ by iterating the construction of non-nested CLE$_\kappa$ in each loop. That is, when $\kappa \in (8/3,4]$, at each iteration, one constructs a non-nested CLE$_\kappa$ in the interior of each loop drawn at the previous iteration, while when $\kappa \in (4,8)$, one constructs a non-nested CLE$_\kappa$ in the interior of each connected component of each loop drawn at the previous iteration.

Here and in the sequel, when we say that $\Gamma$ is a non-nested CLE$_\kappa$, $\Gamma$ refers to the set defined by the closure of the union of all the loops. However, when we say that $\Gamma$ is a nested CLE$_\kappa$, we mean the collection of all loops -- the closure of the union of all nested loops would just give the full domain. There are many ways to put a topology on this collection of nested loops to obtain a metrizable space (for example, extending the definition in \cite[Section 2.1]{CLE_Markovian} to also keep track of the nesting generation).

An important property of non-nested and nested CLE$_\kappa$ is their conformal invariance in law: if $\varphi: D \to \tilde D$ is a conformal map between two open and simply connected domains of $\mathbb{C}$ and $\Gamma$ is a non-nested, resp. nested, CLE$_\kappa$ in $D$, then $\varphi(\Gamma)$ has the law of a non-nested, resp. nested, CLE$_\kappa$ in $\tilde D$.

The GFF with Dirichlet boundary conditions and  nested CLE$_4$ are deeply connected, \cite{MSCLE}. As this result is at the core of the proof, let us now provide some more details; see also \cite[Section~4]{BTLS}. Set $\lambda:= \sqrt{\pi/8}$  and let $h$ be a Dirichlet GFF in $\mathbb{D}$. Then a non-nested CLE$_4$ $\tilde \Gamma$ in $\mathbb{D}$ can be coupled to $h$ as a so-called local set \cite{localset, IG_1}. The key point is that in this coupling, conditionally on $\tilde \Gamma$, we can decompose
\begin{equation*}
    h = \sum_{j} h^{O_j} + H
\end{equation*}
where the $h^{O_j}$ are independent GFFs with Dirichlet boundary conditions in each simply connected component $O_j$ of $\mathbb{D} \setminus \tilde \Gamma$ and $H$ is a random distribution in $\mathbb{D}$ that is almost surely constant when restricted to each $O_j$, with $\PP(H=-2\lambda \text{ in $O_j$})=\PP(H=2\lambda \text{ in $O_j$})=1/2$ independently for each $j$. Moreover, the fields $h^{O_j}$ and $H$ are independent. This coupling can then be iterated in each $O_j$ with respect to the field $h^{O_j}$ and this eventually gives rise to a coupling between $h$ and a \emph{nested} CLE$_4$ $\Gamma$ in $\mathbb{D}$.

{Another way of defining this coupling is as follows. Let $\Gamma$ be a nested CLE$_4$ in $\mathbb{D}$. For $n \geq 1$, let $\text{Loop}^{(n)}(\Gamma)$ denote the set of loops drawn at iteration $n$ of $\Gamma$ and for $z \in \mathbb{D}$, let $\ell_{z}^{n}$ be the loop of $\text{Loop}^{(n)}(\Gamma)$ surrounding $z$. Define also, for $z \in \mathbb{D}$,
\begin{equation} \label{def_HN}
    H_n(z) = \sum_{j=1}^{n} \xi_{\ell_{z}^{j}}
\end{equation}
where $(\xi_{\ell_{z}^{j}})_{1\leq j \leq n}$ are independent and identically distributed random variables, one for each loop surrounding $z$ until iteration $n$, with for $1 \leq j \leq n$, $\PP(\xi_{\ell_{z}^{j}} = -2\lambda)=\PP(\xi_{\ell_{z}^{j}} = 2\lambda)=1/2$. As $n \to \infty$, the function $z \mapsto H_n(z)$ almost surely converges in the space of distributions to a GFF $h$ with Dirichlet boundary conditions in $\mathbb{D}$, \cite{MSCLE,nesting_lim,BTLS}. In the resulting coupling, for $n \geq 1$, conditionally on all the loops 
and their labels 
up to iteration $n$, the field $h$ can be decomposed as
\begin{equation*}
    h = \sum_{j} h^{O_j} + H_n
\end{equation*}
where the $h^{O_j}$ are independent Dirichlet GFFs in the interior $O_j$ of each loop of $\text{Loop}^{(n)}(\Gamma)$, and $H_n$ is the function defined in \eqref{def_HN}. The fields $h^{O_j}$ and $H_n$ are independent and note that $H_n$ takes values in $\{-2n\lambda, \dots, 2n\lambda\}$ and is constant when restricted to each $O_j$. Moreover, it can be shown that $\Gamma$ and the labels $(\xi_{\ell}, \ell \in \text{Loop}(\Gamma))$ of the loops of $\Gamma$ are deterministic functions of the field $h$ \cite{BTLS}.}

In view of this coupling, heuristically, points in $\Gamma$ should almost surely \emph{not} be thick points of $h$. Therefore, understanding the connectedness properties of $\mathcal{T}_{\gamma}(h)$ amounts to understanding those of the complement of the nested CLE$_4$. Provided this heuristic can be made precise, the other key ingredient needed to prove Theorem \ref{th_thickpts} is the following result, which asserts that the complement of a nested CLE$_{4}$ is totally disconnected. In fact, since there is no extra work involved, we prove the result for arbitrary CLE$_\kappa$ with $\kappa \leq 4$.

\begin{theorem} \label{th_CLE}
Let $\kappa \in (8/3,4]$ and let $\Gamma$ be a nested CLE$_\kappa$ in $\mathbb{D}$. Then the complement of $\Gamma$, i.e., the complement in $\D$ of the union of all loops in $\Gamma$, is almost surely totally disconnected.
\end{theorem}

The proof of Theorem \ref{th_CLE} uses the coupling between a non-nested CLE$_{\kappa}$ and a Brownian loop soup with intensity $c=c(\kappa)$ \cite{CLE_Markovian} and this is why the result is stated only for $\kappa \in (8/3,4]$. Section \ref{sec_CLE} will be dedicated to its proof.

With Theorem \ref{th_CLE} in hand, the proof of Theorem \ref{th_thickpts} reduces to showing that thick points of the Dirichlet GFF must be contained in the complement of its coupled nested CLE$_4$ loops. We will in fact show something slightly stronger, as already explained informally in Theorem \ref{th_intro_nestingfield}. Namely, that the set of thick points agrees with the set of points where $H_n$, defined in \eqref{def_HN}, grows atypically fast. In particular this means that it cannot include any points on the nested CLE$_4$ loops themselves.

Let us now define this latter set more precisely. It is one example of the set of ``thick points'' of a so-called weighted CLE$_\kappa$ nesting field, introduced and studied in \cite{nesting_extreme, nesting_lim}. So let us make a small detour to explain the general construction. Let $\kappa \in (8/3,8)$ and let $\Gamma$ be a nested CLE$_{\kappa}$ in $D$. Fix a probability distribution $\mu$ on $\mathbb{R}$ with mean $0$ and finite second moment. Conditionally on $\Gamma$, let $(\xi_{\ell})_{\ell \in \Gamma}$ be i.i.d random variables with distribution $\mu$. The associated (truncated) weighted nesting field is defined as, for $z \in D$ and $r>0$,
\begin{equation} \label{def_nestingfield}
    S_r(z) := \sum_{\ell \in \Gamma_r(z)} \xi_{\ell}
\end{equation}
where $\Gamma_r(z)$ is the set of loops in $\Gamma$ that surrounds the ball $B(z,r)$. In particular, if the point $z$ lies on one of the loops of $\Gamma$, then $S_r(z)$ will be constant above some finite value of $r$. {Comparing to above, e.g.\ around \eqref{def_HN}, we see that in the case where $\kappa = 4$ and $\mu$ is Rademacher, the limiting field is just a multiple of the GFF. As shown in \cite{nesting_lim}, this object has a limit more generally as $r \to 0$:} for any fixed $\delta>0$, there exists a $H^{-2-\delta}_{\text{loc}}(D)$--valued random variable $S$ such that for all $f \in \mathcal{C}_{c}^{\infty}(D)$, almost surely $\lim_{r \to 0} \langle S_r, f \rangle = \langle S, f \rangle$. As in the case of the Gaussian free field, $S$ is not defined pointwise but does admit a special set of points, called thick points, at which loosely speaking, $S$ takes atypically high or low values. More precisely, let $\alpha \in \mathbb{R}$. The set of $\alpha$-thick points naturally associated to $S$ is defined to be 
\begin{equation} \label{def_thickCLE}
  \Phi_{\alpha}^{\mu}(\Gamma):=\{ z \in D: \lim_{r \to 0} \frac{S_r(z)}{\log 1/r} = \alpha \}.
\end{equation}

The set of values of $\alpha$ for which $\Phi_{\alpha}^{\mu}(\Gamma)$ is almost surely non-empty depends on $\kappa$ and $\mu$. When $\Phi_{\alpha}^{\mu}(\Gamma)$ is almost surely non-empty, its almost sure Hausdorff dimension is given in terms of the Fenchel-Legendre transforms of $\mu$ and of the law of the difference in log conformal radii between two successive CLE$_\kappa$ loops \cite{nesting_extreme}. Moreover, $\Phi_{\alpha}^{\mu}(\Gamma)$ is conformally invariant, in the sense that if $\varphi$ is a conformal map from $D$ to another simply connected domain, then $\Phi_{\alpha}^{\mu}(\varphi(\Gamma)) = \varphi(\Phi_{\alpha}^{\mu}(\Gamma))$ almost surely. The next result, which follows almost immediately from Theorem \ref{th_CLE}, says that this set is almost surely totally disconnected when $\kappa \in (8/3,4]$ and $\alpha \neq 0$. By conformal invariance, as before, we may restrict ourselves to $D=\mathbb{D}$.

\begin{corollary} \label{th_nestingfield}
Let $\kappa \in (8/3,4]$ and let $\Gamma$ be a nested CLE$_{\kappa}$ in $\mathbb{D}$. Let $\mu$ be a probability distribution on $\mathbb{R}$ such that $\mu$ has $0$ mean and finite second moment. Then 
\begin{equation*}
    \PP( \Phi_{\alpha}^{\mu}(\Gamma) \text{ is totally disconnected } \forall \alpha\in \mathbb{R}\setminus \{0\}) = 1.
\end{equation*}
\end{corollary}

To complete the proof of Theorem \ref{th_thickpts}, observe that {as mentioned in the discussion before Theorem \ref{th_CLE} describing the coupling of a Dirichlet GFF $h$ with a nested CLE$_4$}, $h$ \emph{is} the limiting weighted nesting field associated to the CLE, with labels given by the $\xi_{l_z^j}$s from \eqref{def_HN}.
In other words, if for $z \in \mathbb{D}$ and $r>0$ we set 
\begin{equation} \label{def_nesting4}
   S_r(z) = \sum_{\ell \in \Gamma_{r}(z)} \xi_{l_z^{n(\ell)}} = H_{J_{z,r}^{\cap}-1}(z)
\end{equation}
where $n(\ell)$ is such that $\ell \in \text{Loop}^{(n(\ell))}(\Gamma)$ and $J_{z,r}^{\cap}$ is the nesting depth of the first loop in $\Gamma$ intersecting $B(z,r)$, {then $z \to S_r(z)$ defines a weighted nesting field}. In what follows, we refer to this special instance of $(S_r(z))_{r,z}$ as the weighted CLE$_4$ nesting field coupled to $h$. We write $\Phi_\gamma(h)$ for the associated set of $\gamma$-thick points as in \eqref{def_thickCLE}. Observe that Corollary \ref{th_nestingfield} implies that $\Phi_\gamma(h)$ is almost surely totally disconnected for all $\gamma\in \mathbb{R}\setminus\{0\}$ since the distribution of the labels in this case is given by $\mu(\{2\lambda\}) =\mu(\{-2\lambda\})=1/2$, which is indeed centered with finite second moment.

Given Corollary \ref{th_nestingfield}, the following relation between $h$ and $S$ will allow us to conclude the proof of Theorem \ref{th_thickpts}. Here and in the sequel, we set for $z \in \mathbb{D}$ and $r>0$,
\begin{equation} \label{def_tilde}
    \tilde h_r(z) = \frac{h_r(z)}{\log 1/r} \quad \text{and} \quad \tilde S_r(z) = \frac{S_r(z)}{\log 1/r},
\end{equation}
where as usual, we work with the jointly H\"{o}lder continuous version of the circle average process $(h_r(z))_{r,z}$ described in the discussion around \eqref{eq_versionintro}.

\begin{theorem} \label{th_CLEGFF}
Let $h$ be a GFF with Dirichlet boundary conditions in $\mathbb{D}$, 
let $(h_r(z))_{r,z}$ be its circle average field and let $(S_r(z))_{r,z}$ be its coupled weighted nesting CLE$_4$ field. Then 
\begin{equation*}
    \sup_{z \in \mathbb{D}} \, \limsup_{r \to 0} \, \vert \tilde h_{r}(z) - \tilde S_r(z) \vert = 0 \quad \text{almost surely.}
\end{equation*}
\end{theorem}

In particular, Theorem \ref{th_CLEGFF} implies Theorem \ref{th_intro_nestingfield}.

Finally, from Theorem \ref{th_CLEGFF}, we deduce a side result about the thickness of a special class of local sets for the GFF, called bounded-type local sets (BTLS), which were introduced in \cite{BTLS} and further studied in \cite{TVS}. Recall that if $A$ is a local set coupled to a GFF $h$, then conditionally on $A$, $h=h_A+H_A$ where $A$ is a GFF with Dirichlet boundary conditions in $D \setminus A$ and $H_A$ is almost surely a harmonic function when restricted to $D \setminus A$ \cite[Section~4.2]{book_GFF}. BTLS are then defined as follows. Let $h$ be a Dirichlet GFF in an open and simply connected domain $D \subset \mathbb{C}$. A set $A$ is said to be a BTLS coupled to $h$ if
\begin{itemize}
    \item there exists a constant $K > 0$ such that almost surely $\vert H_A \vert \leq K$ in $D \setminus A$;
    \item $A$ is a thin local set, that is for all $f \in \mathcal{C}_c^{\infty}(D)$, conditionally on $A$, $(h,f)=(h_A, f)+\int_{D \setminus A} H_A(x)f(x)dx$;
    \item almost surely each connected component of $A$ that does not intersect $\partial D$ has a neighbourhood that does intersect no other connected component of $A$.
\end{itemize}
The two-valued sets of a GFF $h$ are particular examples of BTLS: for $a, b >0$ such that $a+b\geq 2\lambda$, the two-valued set $\TVS_{-a,b}$ is the only local set coupled to $h$ such that the corresponding harmonic function $H_{\TVS_{-a,b}}$ takes values in $\{-a,b\}$ \cite{TVS}.

\begin{corollary} \label{cor_BTLS}
Let $K > 0$. Let $h$ be a GFF in $\mathbb{D}$ with Dirichlet boundary conditions and let $A$ be a $K$-BTLS coupled to $h$. Then, for any $\gamma \in [-2,2]\setminus \{0\}$,
\begin{equation*}
    \{ z \in \mathbb{D}: \lim_{r \to 0} \frac{h_r(z)}{\log 1/r} = \frac{\gamma}{\sqrt{2\pi}}, z \in A \} = \emptyset \quad \text{almost surely.}
\end{equation*}
\end{corollary}

This paper is structured as follows. In Section \ref{sec_CLE}, we prove Theorem \ref{th_CLE} and Corollary \ref{th_nestingfield}. Then, in Section \ref{sec_GFF}, we turn to the proof of Theorem \ref{th_CLEGFF}, thus completing the proof of Theorem \ref{th_thickpts}. This section ends with the proofs of Proposition \ref{prop_singular} and Corollary \ref{cor_BTLS}.

\subsection*{Acknowledgements} The authors would like to thank N.~Berestycki for many interesting and fruitful discussions. J.~Aru was supported by Eccellenza grant 194648 of the Swiss National Science Foundation. L.~Papon was supported by an EPSRC doctoral studentship. E.~Powell was supported in part by UKRI future leaders fellowship MR/W008513/1.

\section{The complement of nested CLE is almost surely totally disconnected} \label{sec_CLE}

Let $\kappa \in (8/3,4]$ and let $\Gamma$ be a nested CLE$_{\kappa}$ in $\mathbb{D}$. For $n \in \mathbb{N}$, we say that a loop $\ell \in \Gamma$ has depth $n$ if it is surrounded by exactly $n-1$ loops of $\Gamma$ and we denote by $\text{Loop}^{(n)}(\Gamma)$ the set of such loops. We set $\Gamma^{(1)} = \overline{\cup_{\ell \in \text{Loop}^{(1)}(\Gamma)} \ell}$ and iteratively define
\begin{equation*}
    \Gamma^{(n+1)} = \overline{\Gamma^{(n)} \cup \cup_{\ell \in \text{Loop}^{(n+1)}(\Gamma)} \, \ell}.
\end{equation*}
The following properties of nested CLE$_\kappa$, $\kappa \in (8/3,4]$, were established in \cite{CLE_Markovian}. For each $n$, $\mathbb{D} \setminus \Gamma^{(n)}$ {almost surely} consists of infinitely many open and simply connected components. These components are the interiors $(\text{int} \, (\ell_j))_j$ of the loops of $\text{Loop}^{(n)}(\Gamma)$. In particular, their boundaries are {almost surely} continuous simple loops in $\mathbb{D}$. These loops {almost surely} do not intersect each other and if $\ell_1$ is a loop of depth $n$ and $\ell_2$ a loop of depth $n+1$ surrounded by $\ell_1$, then $\ell_2$ {almost surely} does not intersect $\ell_1$. Moreover, since CLE$_\kappa$ is locally finite, for each $n$ and $\epsilon > 0$, {almost surely} only finitely many connected components of $\mathbb{D} \setminus \Gamma^{(n)}$ have diameter larger than $\epsilon$. In the sequel, the number of such connected components will be a quantity of interest and it is convenient to introduce the following notation: if $D \subset \mathbb{C}$ is a simply connected domain,  $U$ a subset of $D$ and $0<x<\text{diam}(D)$, $N_x(U)$ 
denotes the number of connected components of $D \setminus {U}$ that have diameter larger than $x$. In particular, if $\tilde \Gamma$ is a non-nested CLE$_\kappa$ in $D$ and $0<x<\text{diam}(D)$, $N_x(\tilde \Gamma)$ stands for the number of connected components of $D \setminus \tilde \Gamma$ that have diameter larger than $x$. 

Theorem \ref{th_CLE} will follow from the following lemma.

\begin{lemma} \label{lemma_eps}
Let $\kappa \in (8/3,4]$ and let $\Gamma$ be a nested CLE$_{\kappa}$ in $\mathbb{D}$. For any $\epsilon > 0$, there {almost surely} exists $n \in \mathbb{N}$ such that {$N_\eps(\Gamma^{(n)})=0.$}
\end{lemma}

{Assuming this lemma, it is straightforward to} deduce Theorem \ref{th_CLE}.

\begin{proof}[Proof of Theorem \ref{th_CLE} {given Lemma \ref{lemma_eps}}]
Let $\kappa \in (8/3,4]$ and let $\Gamma$ be a nested CLE$_{\kappa}$ in $\mathbb{D}$. We have
\begin{align*}
    \PP(\text{the complement of } \Gamma \text{ is not totally disconnected}) & 
    \leq {\PP(\exists \epsilon > 0 : N_\eps(\cup_{m \in \mathbb{N}^{*}} \Gamma^{(m)})\ne 0)} \\
    &= 0
\end{align*}
where the last {equality} follows from Lemma \ref{lemma_eps} and the fact that $N_\eps(\cup_{m\in \mathbb{N}^*} \Gamma^{(m)})\le N_\eps(\Gamma^{(n)})$ for any fixed $n$. This completes the proof of Theorem \ref{th_CLE} assuming Lemma \ref{lemma_eps}.
\end{proof}

We now turn to the proof of Corollary \ref{th_nestingfield}, which is a consequence of {Theorem \ref{th_CLE}} and of the definition of the set of thick points of a weighted nesting CLE$_\kappa$ field.

\begin{proof}[Proof of Corollary \ref{th_nestingfield}]
For any $a>0$, let $\Phi_{a}^{\mu,+}:=\{z\in \mathbb{D}: \liminf_{r\to 0} \tilde{S}_r(z) \ge a\}$. Then since any $z\in \cup_{m\in \mathbb{N}^*}\Gamma^{(m)}$ has $\liminf_{r\to 0} \tilde{S}_r(z)=0$ by definition, we have that $\Phi_{a}^{\mu,+}\subseteq \D \setminus \cup_{m\in \mathbb{N}^*} \Gamma^{(m)}$. It therefore follows from Theorem \ref{th_CLE} that the event $A_a^+=\{\Phi_a^{\mu,+} \text{ is totally disconnected} \}$ has probability one. The same conclusion holds if we define, {for any $a<0$,} $A_a^-$ analogously, with $\Phi_{a}^{\mu,-}:=\{z\in \mathbb{D}: \limsup_{r\to 0} \tilde{S}_r(z) \le a\}$. Writing the event in the statement of the theorem as $\cup_{n\in \mathbb{N}^*} A_{1/n}^+\cup A_{1/n}^-$, we deduce the result.
\end{proof}

We now turn to the proof of Lemma \ref{lemma_eps}. The idea is to encode the large loops of a nested CLE$_{\kappa}$, $\kappa \in (8/3,4]$, into a tree. For $\epsilon > 0$, the 
vertices of this tree will be the loops of $\Gamma$ that have diameter larger than $\epsilon$ (and two vertices will be connected by an edge if and only if the two corresponding loops differ by exactly one level of nesting and one surrounds the other).  Showing Lemma \ref{lemma_eps} will then amount to show that this tree {almost surely} has finite depth. To carry out this strategy, we will need to estimate quantities of the type $\EE[N_{d}(\tilde \Gamma)]$, where $\tilde \Gamma$ is a non-nested CLE$_\kappa$ and $d>0$. This requires {us} to understand how large loops in a non-nested CLE$_{\kappa}$ are formed, and this is where the restriction to $\kappa \in (8/3,4]$ plays a crucial role: we relate large loops in a non-nested CLE$_{\kappa}$ to crossing events in a Brownian loop-soup of intensity $c(\kappa)$. This strategy would not allow us to deal with the case $\kappa \in (4,8)$ as the Brownian loop-soup construction of non-nested CLE$_\kappa$ does not extend to these values of $\kappa$ \cite{CLE_Markovian}. The following two auxiliary lemmas provide the properties of $\EE[N_{d}(\tilde \Gamma)]$, $d>0$, that will be instrumental in the proof of Lemma \ref{lemma_eps}. 

\begin{lemma} \label{lemma_dbound}
Let $\kappa \in (8/3,4]$. There exists $d=d(\kappa) > 1.9$ such that the following holds. For each $p \in \mathbb{N}^{*}$, there exists a constant $C_p=C_p(\kappa) < \infty$ such that if {$D\subseteq \mathbb{D}$} is an open and simply connected domain 
with $\text{diam}(D)>d$, 
and $\tilde \Gamma_D$ is a non-nested CLE$_{\kappa}$ in $D$, then
\begin{equation*}
    \EE[N_d(\tilde \Gamma_D)^p] \leq C_p.
\end{equation*}
{In particular, the constant $C_p$ does not depend on $D$.}
\end{lemma}

\begin{proof}
The proof relies on the loop-soup construction of the non-nested CLE$_{\kappa}$, $\kappa \in (8/3,4]$. Let $\kappa \in (8/3,4]$. By Theorem 1.6 in \cite{CLE_Markovian}, there exists a unique $c=c(\kappa) \in (0,1]$ such that a non-nested CLE$_{\kappa}$ $\tilde \Gamma$ can be constructed from a Brownian loop-soup $\LL$ with intensity $c$. More precisely, if $D$ is an open and simply connected domain in $\mathbb{C}$, we can couple a non-nested CLE$_{\kappa}$ $\tilde \Gamma_{D}$ in $D$ and a Brownian loop-soup $\LL_{D}$ in $D$ with intensity $c$ in $D$ in such a way that the loops of $\tilde \Gamma_{D}$ correspond to the outermost boundaries of the outermost clusters of loops in $\LL_{D}$. Using this coupling, we are going to show that there exists $d=d(\kappa)$ such that for any $D \subset \mathbb{D}$, the probability that there exist $j$ loops with diameter larger than $d$ in a non-nested CLE$_\kappa$ $\tilde \Gamma_D$ in $D$ decays exponentially fast, uniformly in $D \subset \mathbb{D}$. From this, we will be able to show that the moments of $N_d(\tilde \Gamma_D)$ are uniformly bounded in the domain $D \subset \mathbb{D}$.

Consider the coupling $(\tilde \Gamma_{\mathbb{D}}, \LL_{\mathbb{D}})$ described above {with} $D=\mathbb{D}$. For $w \in \partial \mathbb{D}$ and $0<r_1<r_2<1$, we let
\begin{equation*}
    A(w,r_1,r_2) = \{ z \in \mathbb{D}: r_1 < \vert z-w \vert < r_2 \}
\end{equation*}
denote the "boundary annulus" 
centered at $w$ (on the boundary of $\mathbb{D}$) with inner radius $r_1$ and outer radius $r_2$. For $k \in \mathbb{N}^{*}$ and a loop-soup $\LL$, we also define the event
\begin{equation*}
    \mathcal{C}(\LL; k,w,r_1,r_2) = \{ \text{there exist $k$ disjoint
 chains of loops in $\LL$ that cross $A(w,r_1,r_2)$} \}
\end{equation*}
where a chain of loops in $\LL$ is defined as a sequence $(\ell_k)_k$ of loops in $\LL$ such $\ell_k \cap \ell_{k+1} \neq \emptyset$ {for each $k$}. By the results in \cite{CLE_Markovian}, the loops of $\LL_{\mathbb{D}}$ {almost surely} do not intersect $\partial \mathbb{D}$. This implies that for any $r_2 < 1$ and $w \in \partial \mathbb{D}$,
\begin{equation} \label{cvg_ann}
    \PP(\mathcal{C}(\LL_{\mathbb{D}}; 1, w,\frac{1}{n}, r_2)) \to 0 \quad \text{as } n \to \infty.
\end{equation}
Indeed, the sequence $\mathcal{C}(\LL_{\mathbb{D}}; 1, w,\frac{1}{n}, r_2)_{n \geq 1}$ is a decreasing sequence of events and if 
the intersection over $n$ occurs, there must exist a cluster of loops in $\LL_{D}$ whose closure intersects the boundary of $\mathbb{D}$. By \cite[Lemma~9.4]{CLE_Markovian}, this event has zero probability. Since $\lim_{n \to \infty} \PP(\mathcal{C}(\LL_{\mathbb{D}}; 1, w,\frac{1}{n}, r_2)) = \PP(\lim_{n \to \infty} \mathcal{C}(\LL_{\mathbb{D}}; 1, w,\frac{1}{n}, r_2))$, we deduce that the convergence \eqref{cvg_ann} holds.

Let us now fix $r_2 < 1$ and $w \in \partial \mathbb{D}$. From the previous convergence, we deduce that there exists $r_1 \in (0,r_2)$ such that $\PP(\mathcal{C}(\LL_{\mathbb{D}}; 1,w,r_1,r_2))<1$. Moreover, by conformal (rotational) invariance of the Brownian loop-soup, $r_1$ does not depend on $w$. Hence, we may set $p_{\mathbb{D}}:=\PP(\mathcal{C}(\LL_{\mathbb{D}}; 1,w,r_1,r_2))$, for this particular choice of $r_1$, $r_2$ (and arbitrary $w\in \partial \mathbb{D}$).

With $r_1$ and $r_2$ as above, we are now going to surround $\partial \mathbb{D}$ from the inside of $\mathbb{D}$ by a collection $\mathcal{A}(\mathbb{D})$ of boundary annuli of inner radius $r_1$ and outer radius $r_2$. Pick a point $w_1 \in \partial \mathbb{D}$ and add $A(w_1,r_1,r_2)$ to $\mathcal{A}(\mathbb{D})$. Then let $w_2$ be the most clockwise intersection point of $\partial B(w_1,r_1)$ with $\partial \mathbb{D}$. Add $A(w_2,r_1,r_2)$ to $\mathcal{A}(\mathbb{D})$. Continue this procedure until the newly added boundary annulus intersects the inner radius of $A(w_1,r_1,r_2)$. The obtained collection $\mathcal{A}(\mathbb{D})=\{A(w_j,r_1,r_2)\}_j$ contains $O(r_1^{-1})$ boundary annuli and is such that any boundary point of $\mathbb{D}$ lies inside the inner circle of some $A(w_j,r_1,r_2)$ in $\mathcal{A}(\mathbb{D})$. 

\begin{figure}
      \includegraphics[width=0.5\textwidth]{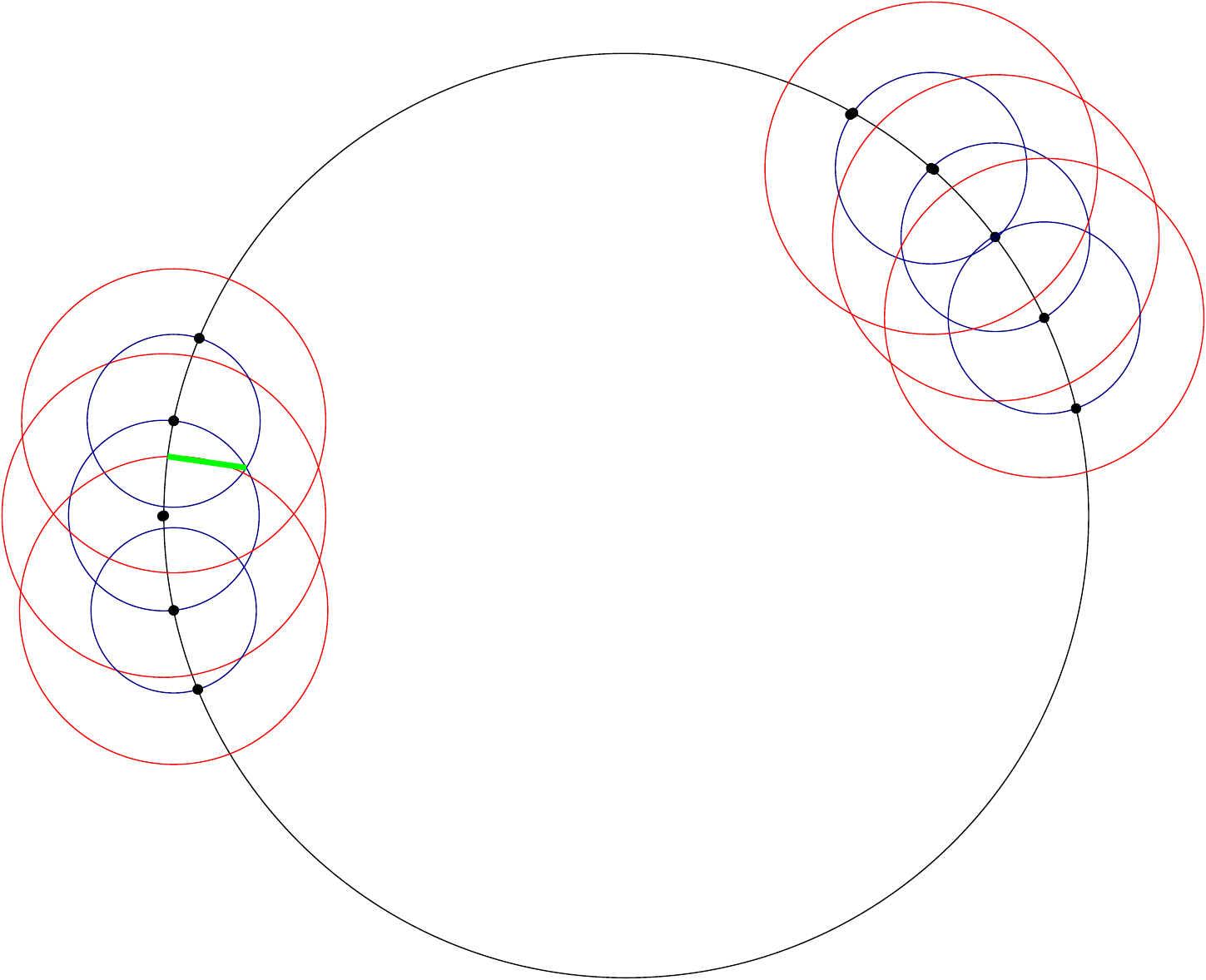}
      \includegraphics[width=0.5\textwidth]{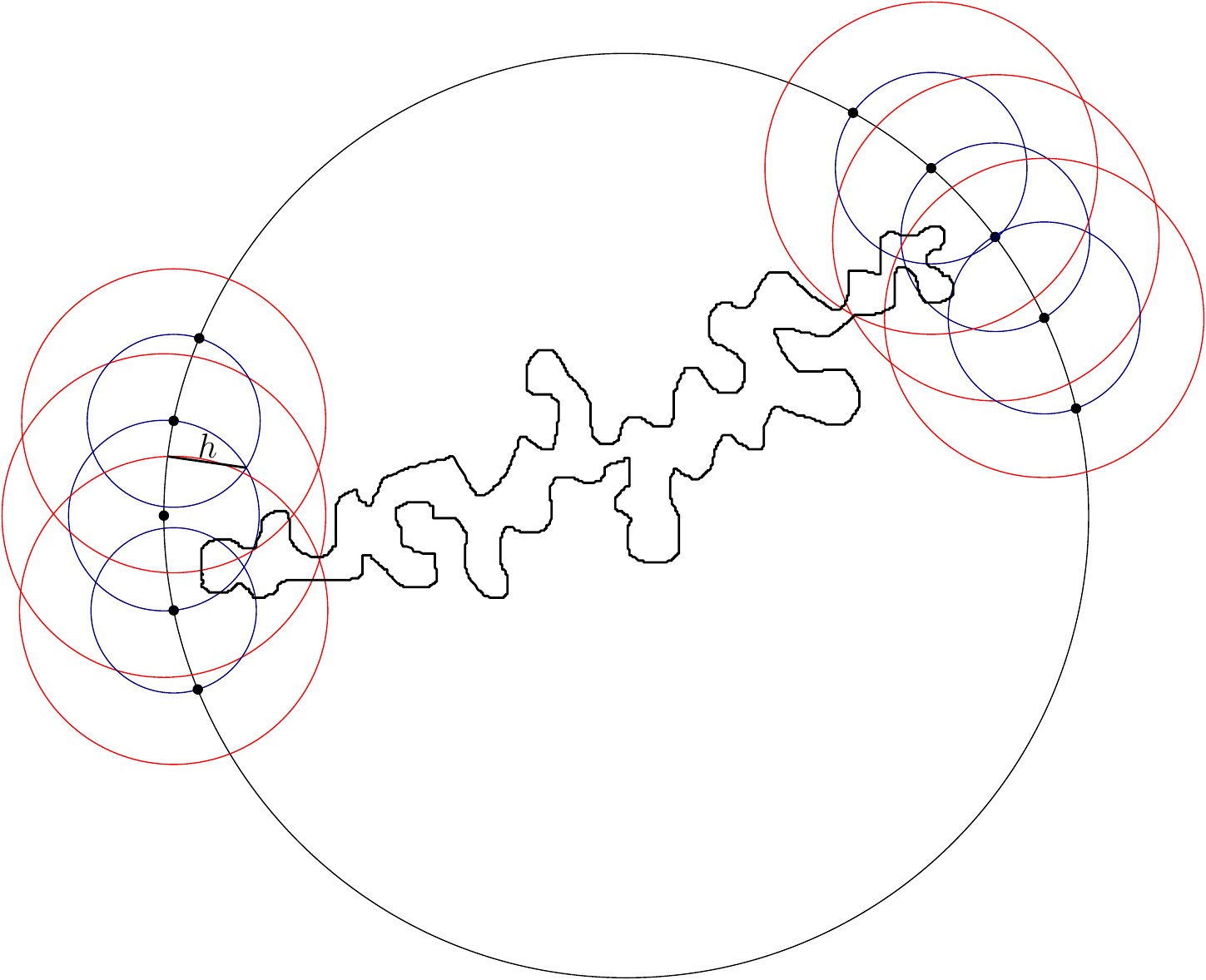}
  \caption{On the left, the unit disk and a subset of the cover $\mathcal{A}(\mathbb{D})$. The red circles have radius $r_2$ and the blue circles $r_1$. The green line represents the distance $h$. On the right, a large $\mathrm{CLE}_{\kappa}$ loop is drawn: it intersects the interior of two different blue circles.}
  \label{fig_cover}
\end{figure}

Let $h$ denote the distance between $\partial \mathbb{D}$ and the intersection points of the inner circles of any two adjacent boundary annuli in $\mathcal{A}(\mathbb{D})$, see Figure \ref{fig_cover}. By construction, $h$ is the same for any pair of adjacent boundary annuli. Set $d= \max\{2-h, 1.9\}$. 

Let $D \subset \mathbb{D}$ be an open and simply connected domain such that $\text{diam}(D) > d$. We define the following covering of $\partial D$
\begin{equation} \label{cover_D}
    \mathcal{A}(D) := \{A = A(w,r_1,r_2 ) \in \mathcal{A}(\mathbb{D}): B(w,r_1) \cap D \neq \emptyset \}.
\end{equation}
Notice that $\vert \mathcal{A}(D) \vert \leq \vert \mathcal{A}(\mathbb{D}) \vert$. We couple a non-nested CLE$_{\kappa}$ $\tilde \Gamma_{D}$ in $D$ and a Brownian loop-soup of intensity $c(\kappa)$ in such a way $\tilde \Gamma_D$ corresponds to the outermost boundaries of the outermost clusters of loops in $\LL_D$. Observe that a loop in $\tilde \Gamma_{D}$ with diameter larger than $d$ must intersect at least two different boundary annnuli in $\mathcal{A}(D)$ and must actually cross one of these two boundary annuli.  Since, in the coupling, the loops of $\tilde \Gamma_{D}$ are the outer boundaries of the outermost clusters of loops of $\LL_{D}$, the existence of such a loop in $\tilde \Gamma_{D}$ implies that a chain of loops in $\LL_{D}$ crosses a boundary annulus in $\mathcal{A}(D)$. Moreover, if there are $j$ disjoint CLE$_4$ loops in $\tilde \Gamma_{D}$, then there must $j$ disjoint chains of loops in $\LL_{D}$: these $j$ disjoint CLE$_4$ loops correspond to the outer boundaries of $j$ disjoint outermost clusters of $\LL_{D}$. To upper bound the probability that there exist such $j$ disjoint clusters in $\LL_{D}$, we observe that for $j \geq 1$, by the pigeon hole principle,
\begin{equation*}
    \{ \exists j \, \text{disjoint clusters in } \LL_{D} \} \subset \{ \exists A \in \mathcal{A}(D) \, \text{such that }  \mathcal{C}(\LL_{D};\lceil \frac{j}{\vert \mathcal{A}(D) \vert}\rceil, w,r_1, r_2) \, \text{occurs}\}.
\end{equation*}

Therefore, for $j \geq 1$, by a union bound,
\begin{align} \label{ineq_PdiamCLE}
    \PP(N_{d}(\tilde \Gamma_{D}) \geq j) & \leq \PP \bigg(\bigcup_{A=A(w,r_1,r_2) \in \mathcal{A}(D)}  \mathcal{C}(\LL_{D}; \lceil \frac{j}{\vert \mathcal{A}(D) \vert}\rceil, w,r_1, r_2) \bigg) \nonumber \\
    &\leq \sum_{A=A(w,r_1,r_2) \in \mathcal{A}(D)} \PP \bigg(\mathcal{C}(\LL_{D};\lceil \frac{j}{\vert \mathcal{A}(D) \vert}\rceil, w,r_1, r_2) \bigg).
\end{align}
To upper bound the probabilities appearing on the right hand side of (\ref{ineq_PdiamCLE}), we apply the BK inequality for Poisson point processes \cite{BK_ineq}. This allows us to upper bound the probability that there exist $k$ crossings of a given boundary annulus by disjoint chains of loops of $\LL_D$ as follows. Given a realisation $\omega$ of the Brownian loop soup, we say that two increasing events $A$ and $B$ occur disjointly, and denote this event by $A \circ B$, if there exist two disjoint subsets $K=K(\omega)$, $L=L(\omega)$  of $D$ such that $\omega_{K} \in A$ and $\omega_{L} \in B$, where $\omega_{K}$, respectively $\omega_{L}$, denotes the realisation $\omega$ restricted to $K$, respectively to $L$. The BK inequality then states that $\PP(A \circ B) \leq \PP(A)\PP(B)$. For $k \geq 1$, the event $\mathcal{C}(\LL_{D}; k,w,r_1,r_2)$ corresponds to the disjoint occurence of $k$ increasing events:
\begin{equation*}
    \mathcal{C}(\LL_{D}; k,w,r_1,r_2) = \mathcal{C}(\LL_{D}; 1,w,r_1,r_2) \circ \dots \circ \mathcal{C}(\LL_{D}; 1,w,r_1,r_2).
\end{equation*}
Therefore, for any $w \in \partial \mathbb{D}$ and $k \in \mathbb{N}^{*}$, 
\begin{equation} \label{BK_D}
    \PP(\mathcal{C}(\LL_{D}; k, w,r_1, r_2)) \leq \PP(\mathcal{C}(\LL_{D}; 1, w,r_1, r_2))^{k} = e^{-c_{D}k}
\end{equation}
where we have set $c_{D}:= -\log(\PP(\mathcal{C}(\LL_{D}; 1, w,r_1, r_2)))$. 

It is also not hard to see that 
\begin{equation} \label{ineq_cd}
    c_D \geq c_{\mathbb{D}} =-\log p_\mathbb{D} > 0.
\end{equation}
Indeed, this follows from the fact that $\LL_D$ and $\LL_{\mathbb{D}}$ can be coupled so that $\LL_D\subseteq \LL_{\mathbb{D}}$ almost surely (due to the restriction property of the Brownian loop measure \cite{BLS}).
In this coupling, if $A=A(w,r_1,r_2) \in \mathcal{A}(D)$, then the existence of a chain of loops in $\LL_D$ that crosses $A$ implies the existence of a chain of loops in $\LL_{\mathbb{D}}$ that crosses $A$.  
 Therefore
\begin{equation*}
    \PP(\mathcal{C}(\LL_D; 1,w,r_1,r_2)) \leq \PP(\mathcal{C}(\LL_{\mathbb{D}}; 1,w,r_1,r_2)) =p_\mathbb{D},
\end{equation*}
which is equivalent to  (\ref{ineq_cd}).

Going back to (\ref{ineq_PdiamCLE}), the inequality (\ref{BK_D}) yields
\begin{equation} \label{ineq_expD}
    \PP(N_{d}(\tilde \Gamma_{D} \geq j)) \leq \vert \mathcal{A}(D) \vert \exp \bigg(-c_{D} \bigg \lceil \frac{j}{\vert \mathcal{A}(D) \vert}  \bigg \rceil \bigg).
\end{equation}
The statement of Lemma \ref{lemma_dbound} follows from this. Indeed, for $p \in \mathbb{N}^{*}$,
\begin{align*}
    \EE[N_d(\tilde \Gamma_D)^p] &= \sum_{k \geq 1} k^p \PP(N_d(\tilde \Gamma_D) = k) \\
    & \leq \sum_{k\geq 1} k^p \PP(N_d(\tilde \Gamma_D) \geq k) \\
    &\leq \sum_{j \geq 0} \, \sum_{k = j \vert \mathcal{A}(D) \vert + 1}^{(j+1) \vert \mathcal{A}(D) \vert} k^p \vert \mathcal{A}(D) \vert e^{-c_{D}(j+1)}
\end{align*}
where we used (\ref{ineq_expD}) in the last inequality. Upper bounding $\vert \mathcal{A}(D) \vert$ by $\vert \mathcal{A}(\mathbb{D}) \vert$ and remembering (\ref{ineq_cd}), we obtain that
\begin{align*}
    \EE[N_d(\tilde \Gamma_D)^p] &\leq  \vert \mathcal{A}(\mathbb{D}) \vert e^{-c_{\mathbb{D}}} \sum_{j \geq 0} (j+1)^p \vert \mathcal{A}(D) \vert^{p+1} e^{-c_{\mathbb{D}}j} \\
    &\leq \vert \mathcal{A}(\mathbb{D}) \vert^{p+2} e^{-c_{\mathbb{D}}} \sum_{j \geq 0} (j+1)^p e^{-c_{\mathbb{D}}j},
\end{align*}
and the series on the right-hand side is finite. This concludes the proof of Lemma \ref{lemma_dbound}.
\end{proof}

The uniform bound on the moments of $N_d(\tilde \Gamma_D)$ established in Lemma \ref{lemma_dbound} allows us to apply H\"older inequality to derive the following result. As one may guess, this result will be the key to show the {almost sure} finiteness of the tree constructed in the proof of Lemma \ref{lemma_eps}.

\begin{lemma} \label{lemma_expectation}
Let $\kappa \in (8/3,4]$ and $d=d(\kappa)$ be as given by Lemma \ref{lemma_dbound}. Then there exists $d_0=d_0(\kappa) \in [d,2)$ and $c<1$ such that the following holds. {Let $D\subseteq \mathbb{D}$} be an open and simply connected domain 
with $\text{diam}(D)>d$. 
Let $\tilde \Gamma_D$ be a non-nested CLE$_{\kappa}$ in $D$. Then

\begin{equation*}
    \EE[N_{d_0}(\tilde \Gamma_D)]\le c<1.
\end{equation*}
In particular, $c<1$ does not depend on $D$.
\end{lemma}

\begin{proof}
Let $D \subset \mathbb{D}$ be an open and simply connected domain such that $\text{diam}(D) > d$ where $d$ is as in Lemma \ref{lemma_dbound}. By H\"older's inequality, for any $\tilde d \in [d,2)$, we have
\begin{equation*}
    \EE[N_{\tilde d}(\tilde \Gamma_{D})] = \EE[N_{\tilde d}(\tilde \Gamma_{D}) \mathbb{I}_{\{N_{\tilde d}(\tilde \Gamma_{D}) > 0\}}] \leq \EE[N_{\tilde d}(\tilde \Gamma_{D})^2]^{1/2} \PP(N_{\tilde d}(\tilde \Gamma_{D}) > 0)^{1/2}.
\end{equation*}
Moreover, since the random variables $(N_{x}(\tilde \Gamma_{D}))$, $x \in (0,2]$, are non-increasing, we can further write, for any $\tilde d \in [d,2)$
\begin{equation} \label{ineq_disk}
    \EE[N_{\tilde d}(\tilde \Gamma_{D})] \leq \EE[N_{d}(\tilde \Gamma_{D})^2]^{1/2} \PP(N_{\tilde d}(\tilde \Gamma_{D}) > 0)^{1/2}.
\end{equation}
By Lemma \ref{lemma_dbound}, there exists a constant $C_2 < \infty$, independent of $D$, such that $\EE[N_{d}(\tilde \Gamma_{D})^2]\leq C_2$. {If $C_2 < 1$, from \eqref{ineq_disk} and the trivial bound $\PP(N_{\tilde d}(\tilde \Gamma_{D}) > 0) \leq 1$, we obtain that any $\tilde d \in [d,2)$, $\EE[N_{\tilde d}(\tilde \Gamma_{D})]<1$ and the lemma follows with $d_0=\tilde d$.} If $C_2 \geq 1$, \eqref{ineq_disk} yields that for any $\tilde d \in [d,2)$
\begin{equation} \label{ineq_disk2}
    \EE[N_{\tilde d}(\tilde \Gamma_{D})] \leq \sqrt{C_2} \PP(N_{\tilde d}(\tilde \Gamma_{D})) > 0)^{1/2}.
\end{equation}
We next claim that for any $\tilde d \in [d,2)$,
\begin{equation} \label{ineq_loopD}  
\mathbb{P}(N_{\tilde d}(\tilde \Gamma_{D})) > 0) \leq \mathbb{P}(N_{\tilde d}(\tilde \Gamma_{\mathbb{D}})) > 0).
\end{equation}
This is easily seen by using a coupling $(\LL_D, \tilde \Gamma_D, \LL_{\mathbb{D}}, \tilde \Gamma_{\mathbb{D}})$ where $(\LL_D,\LL_{\mathbb{D}})$ are coupled as in the explanation to the inequality (\ref{ineq_cd}), and both $(\LL_D,\tilde \Gamma_D)$ and $(\LL_{\mathbb{D}},\tilde \Gamma_\mathbb{D})$ are coupled as in the proof of Lemma \ref{lemma_dbound}. In this coupling, for $\tilde d \in [d,2)$,
\begin{itemize}
    \item if $N_{\tilde d}(\tilde \Gamma_D) = 1$, then the outermost cluster of loops in $\LL_D$ whose outer boundary corresponds to this loop in $\tilde \Gamma_D$ remains the same or gets larger when constructing $\LL_{\mathbb{D}}$. Moreover, during the construction of $\LL_{\mathbb{D}}$, outermost clusters of loops with diameter larger than $\tilde d$ may appear.
    \item if $N_{\tilde d}(\tilde \Gamma_D) \geq 2$, then the corresponding outermost clusters of loops in $\LL_D$ may merge together during the construction of $\LL_{\mathbb{D}}$. In the worst case, all such clusters merge. But the single cluster thus formed still has diameter larger than $\tilde d$.
\end{itemize}
The above reasoning then implies (\ref{ineq_loopD}). Using this in (\ref{ineq_disk2}), we thus obtain that for any $\tilde d \in [d,2)$
\begin{equation} \label{ineq_disk3}
    \EE[N_{\tilde d}(\tilde \Gamma_{D})] \leq \sqrt{C_2} \PP(N_{\tilde d}(\tilde \Gamma_{\mathbb{D}})) > 0)^{1/2}.
\end{equation}
The lemma will therefore follow if we can show that there exists $d_0 \in [d,2)$ such that 
$\PP(N_{\tilde d_0}(\tilde \Gamma_{\mathbb{D}}) > 0)<1/C_2$. But this simply follows from the fact, writing 
$d_\text{max}$ for the diameter of the largest non-nested CLE$_{\kappa}$ loop in $\mathbb{D}$, that 
\begin{equation*}
    \PP(N_{x}(\tilde \Gamma_{\mathbb{D}}) > 0) = \PP(d_{\max} \ge x) \, \downarrow \, 0 \quad \text{as} \quad x \, \uparrow \, 2.
\end{equation*} 
\end{proof}

With Lemma \ref{lemma_expectation} at hand, we can now turn to the proof of Lemma \ref{lemma_eps}.

\begin{proof}[Proof of Lemma \ref{lemma_eps}]
Let $\epsilon > 0$ and let $\Gamma$ be a nested CLE$_{\kappa}$ in $\mathbb{D}$. We are going to encode the large loops in $\Gamma$ into a tree $T(\epsilon)$ that will be constructed progressively during the proof. We start by constructing a tree $T(1;\eps)$ whose root is simply the unit circle $\partial \mathbb{D}$. Recall that for $n \geq 1$, $\Gamma^{(n)}$ denotes the closed union of the first $n$ levels of loops and let $d_0=d_0(\kappa)$ be as in Lemma \ref{lemma_expectation}, i.e. such that $\EE[N_{d_0}(\Gamma^{(1)})] < 1$.  The vertices in the first generation in $T(1;\epsilon)$ are the loops of $\text{Loop}^{(1)}(\Gamma)$  that have diameter larger than $\epsilon$. By local finiteness of CLE$_{\kappa}$, there are {almost surely} finitely many such loops. We then iteratively construct the next generations of $T(1;\epsilon)$ as follows: if a loop $\ell$ corresponds to a vertex at generation $n$ in $T(1;\epsilon)$, and has diameter larger than $d_0$, then its descendants at generation $n+1$ are the loops of $\text{Loop}^{(n+1)}(\Gamma)$ that lie inside $\ell$ and have diameter larger than $\epsilon$. If a loop at generation $n$ has diameter smaller than $d_0$, then it has no descendants at the next generation. Another way to phrase this is that generation $n$ (the $n$th level vertices) of $T(1,\epsilon)$, is the collection of (nesting)-depth $n$ loops in $\Gamma$ with diameter larger than $\epsilon$, and whose parent loops have diameter larger than $d_0$. Notice that by construction, all the loops corresponding to vertices of $T(1;\eps)$ have diameter larger than $\epsilon$. 

Consider now the restriction $T_0(1;\epsilon)$ of $T(1;\epsilon)$ to loops that have diameter larger than $d_0$: that is, $T_0(1,\eps)$ is just $T(1,\eps)$ minus its leaves. $T_{0}(1;\epsilon)$ is equivalently the tree that one would obtain by keeping at each generation only the loops with diameter larger than $d_0$.
Note that if $\ell$ is a loop at generation $n$ of $T_0(1;\epsilon)$, then its number of descendants in $T_0(1;\epsilon)$ at generation $n+1$ is given by $N_{d_0}(\tilde \Gamma_{\ell})$, where $\tilde \Gamma_{\ell}$ is a non-nested CLE$_{\kappa}$ in $\text{int} (\ell)$. Moreover, if $\ell_1$ and $\ell_2$ are two distinct loops at generation $n$ in $T_{0}(1;\epsilon)$, then $N_{d_0}(\tilde \Gamma_{\ell_1})$ and $N_{d_0}(\tilde \Gamma_{\ell_2})$ are independent. By Lemma \ref{lemma_expectation}, for any $n$ and any loop $\ell$ at generation $n$ of $T_0(1;\epsilon)$, since $\ell \subset \mathbb{D}$ with $\text{diam}(\ell) > d_0$, $\EE[N_{d_0}(\tilde \Gamma_{\ell})] \le c <1$. Therefore, $T_0(1;\eps)$ is dominated by a Galton-Watson tree in which the expected number of descendants of each vertex is strictly less than 1. This implies that there {almost surely} exists $k_1$ such that all loops of diameter larger than $d_0$ in $\Gamma$ have depth at most $k_1-1$ in $T_0(1;\eps)$. In other words, there {almost surely} exists $k_1$ such that no connected component of $\mathbb{D} \setminus \Gamma^{(k_1)}$ has diameter larger than $d_0$. By definition, the construction of $T(1;\eps)$ is finished at generation $k_1$, and we define the first part of $T(\eps)$ to be the tree thus obtained.

We then continue the construction of $T(\epsilon)$ starting from the leaves of $T(1;\epsilon)$. These leaves form a collection {$L_1$} of loops that have diameter larger than $\epsilon$ but strictly smaller than $d_0$. Each of these loops belong to a unique $\text{Loop}^{(n)}(\Gamma)$, for some $n \leq k_1$. Define $r_1:= \max \{ \frac{1}{2}\text{diam}(\ell): \ell \in L_1\}$. Notice that $r_1<d_0/2$ {and each loop in $L_{1}$ is {almost surely} contained in some disk of radius $r_1$.} By scale {and translation} invariance of CLE$_{\kappa}$, plus Lemma \ref{lemma_expectation}, if we define  
\begin{equation} \label{scaleinv_d}
    \frac{d_1}{2r_1} = \frac{d_0}{2} < 1,
\end{equation}
then 
\begin{equation}\label{scaleinv_moment}
    \EE[N_{d_1}(\tilde \Gamma_{D})] \le c <1
\end{equation} 
whenever $D$ is a simply connected domain with diameter less than or equal to $2r_1$ and  $\tilde \Gamma_{D}$ has the law of a non-nested CLE$_{\kappa}$ in $D$.

We are now going to grow trees rooted at {each of the leaves of $T(1;\epsilon)$: in other words, at the loops in $L_{1}$, which we enumerate as $l_1,\dots, l_N$ for some $N<\infty$.} Starting from $\ell_{j} \in L_{1}$, we construct a tree $T^{j}(2;\eps)$ as follows. Let $n_j$ be such that $\ell_j \in \text{Loop}^{(n_j)}(\Gamma)$. If a loop $\ell$ at generation $n$ in $T^j(2;\eps)$ has diameter larger than $d_1$, then its descendants at generation $n+1$ are the loops of $\text{Loop}^{(n_j+n+1)}(\Gamma)$ inside $\ell$ that have diameter larger than $\epsilon$. If on the contrary a loop at generation $n$ in $T^j(2;\eps)$ has diameter smaller than $d_1$, then it has no descendants at generation $n+1$. 

Arguing the same way as for $T(1,
\epsilon)$ and using \eqref{scaleinv_moment} together with the iterative construction of the nested CLE$_\kappa$, 
it follows that for each $j$, the tree $T^j(2;\eps)$ is almost surely finite. 
We then glue the trees $T^j(2;\epsilon)$ to $T(1;\epsilon)$ to produce a new tree $T(2;\epsilon)$. Note that $T(2;\epsilon)$ is a finite tree whose vertices at depth $n$ are $n$th level loops in the original nested CLE$_\kappa$ and that $T(2;\epsilon)$ contains all loops of the nested CLE$_\kappa$ with diameter larger than $d_1$.    
In other words, {since $T(2,\eps)$ is finite, there exists some $k_2<\infty$ such that} no connected component of $\mathbb{D} \setminus \Gamma^{(k_2)}$ has diameter larger than $d_1$.

This {procedure} yields a new set {$L_{2}$} of leaves for {$T(2;\epsilon)$}
such that no loop in {$L_{2}$} has diameter larger than $d_1$ and we can now repeat the previous construction, starting from the loops in {$L_{2}$}, with $r_2$ and $d_2$ defined {analogously to $r_1,d_1$.} Iterating this procedure, we obtain a decreasing sequence $(d_n)_n$ of diameters and a decreasing sequence $(r_n)_n$ of radii such that {almost surely} 
for all $n \in \mathbb{N}^{*}$,
\begin{equation*}
    \frac{d_n}{2r_n} = \frac{d_0}{2} < 1.
\end{equation*} 
In particular, $(d_n)_n$ and $(r_n)_n$ are such that {almost surely}
\begin{equation*}
    d_n = d_0r_n \leq d_0\frac{d_{n-1}}{2} = d_0^2\frac{r_{n-1}}{2} \leq \dots \leq 2 \bigg( \frac{d_0}{2} \bigg)^{n}.
\end{equation*}
Therefore, there {almost surely} exists $N$ such that $d_{N} \leq \epsilon$. Let $k_N<\infty$ denote the total depth of the finite tree $T(N;\epsilon):=T(\epsilon)$, constructed after $N$ iterations of the above process.

Then all loops with diameter larger than $\epsilon$ {in the original nested CLE $\Gamma$} have depth at most {$k_N$} in $T(\epsilon)$. In other words, no connected component of $\mathbb{D} \setminus \Gamma^{({k_N})}$ has diameter larger than $\epsilon$. 
Since $\epsilon > 0$ was arbitrary, this concludes the proof.
\end{proof}

\section{The {almost sure} total disconnectedness of the thick points of the GFF and its consequences} \label{sec_GFF}

\subsection{Proof of Theorem \ref{th_CLEGFF}}

As explained in the discussion around \eqref{eq_versionintro} in the introduction, we work with a version of the circle average process $(h_r(z))_{r,z}$  such that with probability one, for every $\alpha \in (0,1/2), \upzeta \in (0,1)$ and $\epsilon>0$ there exists a (random) constant $M=M(\alpha,\upzeta,\epsilon)<\infty$ such that for all $z,w \in \mathbb{D}$ and $s,r\in (0,1]$ with $1/2<r/s<2$,
\begin{equation}\label{eq_version}
    \vert h_r(z) - h_{s}(w) \vert \leq M \bigg( \log \frac{1}{r} \bigg)^{\upzeta} \frac{\vert (z,r) - (w,s)\vert^{\alpha}}{r^{\alpha+\epsilon}}.
    \end{equation}
Our goal is to show that for this version of the circle average process, and for any fixed $\delta > 0$,
\begin{equation}\label{goal_CLEGFF}
    \PP(\sup_{z \in \mathbb{D}} \, \limsup_{r \to 0} \, \vert \tilde h_r(z) - \tilde S_r(z) \vert > \delta)=0
\end{equation}
where $\tilde S$ and $\tilde h$ are the re-scaled versions of $S$ and $h$ defined in \eqref{def_tilde}. Recall also the discussion before Theorem \ref{th_CLE} describing the coupling between $h$ and $S$ (its coupled nesting field), defined in \eqref{def_nesting4}.

So let us fix $\delta > 0$. For $0<r<1$, we set 
\begin{equation*}
    r_n(r) = \big(1-\frac{c}{n(r)}\big)n(r)^{-K}
\end{equation*}
where $n(r)$ is the unique $n \in \mathbb{N}, n \geq 2$, such that $r \in [n^{-K}, (n-1)^{-K})$ and {where $c \geq 1/\sqrt{2}$ and $K >0$ are fixed but arbitrary.} For $n \in \mathbb{N}^{*}$ and $z \in \mathbb{D}$, we denote by $z_n(z)$ the closest point to $z$ in the set $D_n := n^{-(K+1)}\mathbb{Z}^{2} \cap \mathbb{D}$. We bound {the left hand side of \eqref{goal_CLEGFF}} above by a sum of three terms{:}
\begin{align}
    &\PP(\sup_{z \in \mathbb{D}} \limsup_{r \to 0} \vert \tilde h_{r}(z) - \tilde h_{r_n}(z_{n(r)}(z)) \vert > \frac{\delta}{3}) \label{proba_GFF} \\
    &+ \PP(\sup_{z \in \mathbb{D}} \limsup_{r \to 0} \vert \tilde h_{r_n(r)}(z_{n(r)}(z)) - \tilde S_{r_n(r)}(z_{n(r)}(z)) \vert > \frac{\delta}{3}) \label{proba_mix} \\
    &+ \PP(\sup_{z \in \mathbb{D}} \limsup_{r \to 0} \vert \tilde S_{r_n(r)}(z_{n(r)}(z)) - \tilde S_r(z) \vert > \frac{\delta}{3}) \label{proba_nested}.
\end{align}
We will handle each of these terms separately and show that they are all equal to $0$ for any $K>0$ and $c \geq 1/\sqrt{2}$.

For the term \eqref{proba_GFF}, this follows easily from the continuity estimates \eqref{eq_version} for the circle average process. To show that the term \eqref{proba_mix} is equal to $0$, we will exploit the coupling between $h$ and $S$. We will condition on an appropriate depth of the nested CLE$_4$ coupled to $h$ and use this conditioning to reduce the proof via Borel-Cantelli lemma to variance estimates for the conditional circle average process. Finally, to deal with the term \eqref{proba_nested}, we will upper bound it by the probability that well-chosen annuli in $\mathbb{D}$ contain many CLE$_4$ loops surrounding their inner boundary but not intersecting their outer boundary. Using Borel-Cantelli lemma and estimates on the extremal distance between a non-nested CLE$_4$ loop and the boundary of the domain in which the non-nested CLE$_4$ is sampled, we will deduce that the term \eqref{proba_nested} is equal to 0.

\subsubsection{Proof that term (\ref{proba_GFF}) is zero} \label{subsec_GFF} 
The fact that this probability is equal to $0$ for any choice of $c,K$ follows from \eqref{eq_version}, which means that for all $\alpha \in (0,1/2), \upzeta \in (0,1)$ and $\epsilon>0$, 
\begin{equation}\label{eq_version2}
    \vert h_r(z) - h_{r_n(r)}(z_{n(r)}(z)) \vert \leq M(\alpha,\upzeta,\epsilon) \bigg( \log \frac{1}{r} \bigg)^{\upzeta} \frac{\vert (z,r) - (z_{n(r)}(z), r_n(r))\vert^{\alpha}}{r_n(r)^{\alpha+\epsilon}} \quad \forall z\in \mathbb{D}, r\in (0,1].\end{equation}
By our choice of $r_n(r)$, we have
\begin{equation*}
    \vert r - r_n(r) \vert \leq (n(r)-1)^{-K} - r_n(r) \leq Cn(r)^{-(K+1)}
\end{equation*}
for some {absolute} constant $C>0$. Moreover, for $n \in \mathbb{N}^{*}$, $\vert z - z_n(z) \vert \leq n^{-(K+1)}/\sqrt{2}$. Substituting into \eqref{eq_version2}, for an arbitrary choice of $\alpha\in (0,1/2)$ and $\epsilon<\alpha/K$ implies that for all $z \in \mathbb{D}$, 
\begin{equation*} \label{eq_limCA}
    \vert \tilde h_r(z) - \tilde h_{r_n(r)}(z_{n(r)}(z)) \vert \to 0 \quad \text{as } r\to 0.
\end{equation*}
Thus,
\begin{equation*}
    \PP(\sup_{z \in \mathbb{D}} \limsup_{r \to 0} \vert \tilde h_{r}(z) - \tilde h_{r_n(r)}(z_{n(r)}(z)) \vert > \frac{\delta}{3}) = 0,
\end{equation*}
as required.

\subsubsection{Proof that term (\ref{proba_mix}) is zero} To show that the probability in \eqref{proba_mix} is equal to $0$, again for any $c,K$, we are going to use the Borel-Cantelli lemma. As we will explain shortly, this requires us to establish that the sum
\begin{equation*}
    \sum_{n \in \mathbb{N}^{*}} n^{2(K+1)} \max_{z_n \in D_n} \PP(\vert \tilde h_{r_n}(z_n) - \tilde S_{r_n}(z_n) \vert > \frac{\delta}{3})
\end{equation*}
is finite. In turn, controlling the probability appearing in this sum requires us to understand how the variance of the circle average of $h$ behaves when the field is conditioned on its coupled nested CLE$_4$. So let us first examine this in more detail.

Let us fix $z \in \mathbb{D}$ and $r>0$ such that $r < 2\text{dist}(z,\partial \mathbb{D})$. For $j \geq 1$, denote by $\ell_z^{j}$ the loop of $\text{Loop}^{(j)}(\Gamma)$ containing $z$ and define
\begin{equation*}
    J_{z,r}^{\cap}:= \min \{ j \geq 1: \ell_z^{j} \cap B(z,r) \neq \emptyset \}.
\end{equation*}
In other words, $J_{z,r}^{\cap}$ is the (nesting)-depth of the largest loop in $\Gamma$ that contains $z$ and intersects $B(z,r)$. We denote this loop by $\ell^{J_{z,r}^{\cap}}$.
For general $j$, we further denote by $\text{CR}(z,\ell^j_{z})$ the conformal radius of $\text{int}(\ell^j_{z})$ seen from $z$. 

Conditioning on $\Gamma^{(J_{z,r}^{\cap})}$, we have
\begin{equation}\label{eq_conditionalunconditional}
    \PP(\vert \tilde h_r(z) - \tilde S_r(z) \vert > \frac{\delta}{3}) = \EE\bigg[\PP \bigg( \vert \tilde h_r(z) - \tilde S_r(z) \vert > \frac{\delta}{3} \, \vert \, \Gamma^{(J_{z,r}^{\cap})}\bigg)\bigg].
\end{equation}
Note that the circle $\partial B(z,r)$ intersects infinitely many loops in $\Gamma^{(J_{z,r}^{\cap})}$. One may therefore think that controlling the conditional variance of $\tilde h_r(z)$ will be rather technical and that conditioning on $\Gamma^{(J_{z,r}^{\cap}-1)}$ may be a better path to follow. However, 
conditioned on $\Gamma^{(J_{z,r}^{\cap}-1)}$, we have that $h$ inside $\ell^{J_{z,r}^{\cap}-1}$ is a Gaussian free field \emph{conditioned} to have its first CLE$_4$ level line loop around $z$ intersecting $B(z,r)$. This is a somewhat complicated conditioning to control.  
On the other hand, we will see that when conditioning on $\Gamma^{(J_{z,r}^{\cap})}$, the sum of the contributions to the variance of $\tilde h_r(z)$ coming from the fields inside each loop intersecting $B(z,r)$ remains smaller than a deterministic constant which does not depend on $z$ and $r$. 

Here, we would like to emphasize that when we condition on $\Gamma^{(J_{z,r}^{\cap})}$, we mean that we condition on the $\sigma$-algebra $\sigma((\mathrm{Loop}(\Gamma^{(j)}))_{1\le j\le J^\cap_{z,r}})$ generated by the loops of $\Gamma$ up to depth $J_{z,r}^{\cap}$. It can be shown that this $\sigma$-algebra is the same as $\sigma(\Gamma^{J_{z,r}^{\cap}}) = \sigma(\overline{\cup_{1 \leq j \leq J_{z,r}^{\cap}}\mathrm{Loop}(\Gamma^{(j)})})$.

Let us denote by $(O_j)_{j}$ the collection of open and simply connected components of $\mathbb{D} \setminus \Gamma^{(J_{z,r}^{\cap})}$. {By definition of the coupling between $h$ and $\Gamma$,} conditionally on $\Gamma^{(J_{z,r}^{\cap})}$,
\begin{equation*}
    h_r(z) = \sum_{j: O_j \cap \partial B(z,r) \neq \emptyset} h_r^{(j)}(z) + \int H(x) \rho_{r}^{z}(dx)
\end{equation*}
where:
\begin{itemize}
    \item $((h^{(j)})_j, H)$ are independent; 
    \item the $h^{(j)}$ are independent GFFs with Dirichlet boundary conditions in each $O_j$; and 
    \item $H$ is {almost surely}  constant when restricted to each $O_j$, satisfying
    \begin{equation*}
    H = S_r(z) + \xi_{j}
    \end{equation*}
    for each $j$, where the $\xi_j$ are independent of $S_r(z)$ and of each other, {with, independently for each $j$, $\PP(\xi_j=-2\lambda)=\PP(\xi_j=2\lambda)=1/2$.}
\end{itemize} 

Moreover, {$\Gamma^{(J_{z,r}^{\cap})}$ 
has the property that the integral of $H$ inside $\Gamma^{(J_{z,r}^{\cap})}$ with respect to $\rho_{r}^{z}$ is almost surely equal to $0$  (it is what is known as a ``thin local set'' of $h$, {\cite{thinLS}, see also \cite[Section~4.2.5]{book_GFF}}).  Therefore, we have that}
\begin{equation*}
    \big \vert \int H(x) \rho_{r}^{z}(dx) - S_r(z)  \big \vert \leq 2\lambda
\end{equation*}
almost surely, and in turn, {almost surely,}
\begin{align} \label{cond_proba}
    \PP \bigg( \vert \tilde h_r(z) - \tilde S_r(z) \vert > \frac{\delta}{3} \,\, \big \vert \,\, \Gamma^{(J_{z,r}^{\cap})}\bigg) &= \PP \bigg( \big \vert \sum_{j: O_j \cap \partial B(z,r) \neq \emptyset} h^{(j)}_r(z) - \int \xi_{\ell_x} \rho_{r}^{z}(dx) \, \big \vert > \frac{\delta}{3} \log \frac{1}{r} \,\, \big \vert \,\, \Gamma^{(J_{z,r}^{\cap})} \bigg) \nonumber \\
    &\leq \PP \bigg( \big \vert \sum_{j: O_j \cap \partial B(z,r) \neq \emptyset} h^{(j)}_r(z) \big \vert > -2\lambda + \frac{\delta}{3} \log \frac{1}{r} \,\, \big \vert \,\, \Gamma^{(J_{z,r}^{\cap})} \bigg).
\end{align}
Now, the $h_r^{(j)}(z)$ are (conditionally) independent Gaussian random variables. Therefore, bounding the conditional probability in (\ref{cond_proba}) 
amounts to controlling the conditional variance of their sum (and using the elementary bound $\PP(\vert X \vert \geq m) \leq C\exp(-\frac{m^2}{\sigma^2})$ for $X \sim \mathcal{N}(0,\sigma^2)$ and $m>0$).
In other words, we need to understand the random variable 
\begin{equation*}
    \EE \bigg [ \bigg( \sum_{j: O_j \cap \partial B(z,r) \neq \emptyset} h_r^{(j)}(z) \bigg)^2 \, \big \vert \, \Gamma^{(J_{z,r}^{\cap})}\bigg].
\end{equation*}
We will prove the following lemma.

\begin{lemma} \label{lemma_varsum}
We have
\begin{equation*}
    \EE \bigg [ \bigg( \sum_{j: O_j \cap \partial B(z,r) \neq \emptyset} h_r^{(j)}(z) \bigg)^2 \, \big \vert \, \Gamma^{(J_{z,r}^{\cap})}\bigg] \leq \log 4
\end{equation*}
almost surely.
\end{lemma}

Before proving this lemma, let us see how it implies that Term \eqref{proba_mix} is $0$ (for any $c,K$). First, combining it with our elementary Gaussian upper bound, we get that
\begin{equation} \label{ineq_probaAS}
    \PP \bigg( \big \vert \sum_{j: O_j \cap \partial B(z,r) \neq \emptyset} h^{(j)}_r(z) \big \vert > -2\lambda + \frac{\delta}{3} \log \frac{1}{r} \, \vert \, \Gamma^{(J_{z,r}^{\cap})} \bigg) \leq C\exp\bigg(- \frac{(-2\lambda + \frac{\delta}{3} \log \frac{1}{r})^2}{{\log 4}} \bigg) 
\end{equation}
almost surely. In particular, the right-hand side is non-random. Then substituting \eqref{ineq_probaAS} into \eqref{cond_proba}, \eqref{eq_conditionalunconditional} {yields}
\begin{equation}\label{eq_ubwithZ}
    \PP(\vert \tilde h_r(z) - \tilde S_r(z) \vert > \frac{\delta}{3})\le C \,  \exp\bigg(- \frac{(-2\lambda + \frac{\delta}{3} \log \frac{1}{r})^2}{{\log 4}} \bigg).
\end{equation}

If we set $r=r_n$ and $z=z_n$ in the above inequality, we see that the right-hand side decays faster than any power of $n$ as $n \to \infty$, at a rate that can be chosen to be independent of $z_n$. This shows that for any $c,K$, the sum
\begin{equation*}
    \sum_{n \in \mathbb{N}^{*}} n^{2(K+1)} \max_{z_n \in D_n} \PP(\vert \tilde h_{r_n}(z_n) - \tilde S_{r_n}(z_n) \vert > \frac{\delta}{3})
\end{equation*}
is finite, so by the Borel-Cantelli lemma, we conclude that
\begin{equation*}
    \PP(\sup_{z \in \mathbb{D}} \limsup_{r \to 0} \vert \tilde h_{r_n(r)}(z_{n(r)}(z)) - \tilde S_{r_n(r)}(z_{n(r)}(z)) \vert > \frac{\delta}{3}) = 0.
\end{equation*}

It thus remains to prove Lemma \ref{lemma_varsum}.

\begin{proof}[Proof of Lemma \ref{lemma_varsum}]
By independence of the fields $h^{(j)}$ conditionally on $\Gamma^{(J_{z,r}^{\cap})}$, we almost surely have
\begin{equation*}
    \EE \bigg [ \bigg( \sum_{j: O_j \cap \partial B(z,r) \neq \emptyset} h_r^{(j)}(z) \bigg)^2 \, \big \vert \, \Gamma^{(J_{z,r}^{\cap})}\bigg] = \sum_{j: O_j \cap B(z,r) \neq \emptyset} \int_{O_j} G_{O_j}(x,y) \rho_{r}^{z}(dx)\rho_{r}^{z}(dy)
\end{equation*}
where for each $j$, $G_{O_j}$ denotes the Green function in $O_j$. Since the $O_j$ are disjoint, setting
\begin{equation*}
    \tilde O = \bigcup_{j: O_j \cap B(z,r) \neq \emptyset} O_j
\end{equation*}
we have, for all $x,y \in \tilde O$,
\begin{equation*}
    G_{\tilde O}(x,y) = \sum_{j: O_j \cap B(z,r) \neq \emptyset} G_{O_j}(x,y)
\end{equation*}
where $G_{\tilde O}$ denotes the Green function in $\tilde O$ and $G_{O_j}(x,y)=0$ if $x$ or $y$ is not in $O_j$. Since the functions $G_{O_j}$ are non-negative, we can apply Tonelli's theorem to obtain that almost surely
\begin{equation*}
    \sum_{j: O_j \cap B(z,r) \neq \emptyset} \int_{O_j} G_{O_j}(x,y) \rho_{r}^{z}(dx)\rho_{r}^{z}(dy) = \int_{\tilde O} G_{\tilde O}(x,y) \rho_{r}^{z}(dx)\rho_{r}^{z}(dy).
\end{equation*}
By monotonicity of the Green function, see for example \cite{book_GFF}, we then almost surely have
\begin{equation} \label{niceCA}
    \int_{\tilde O} G_{\tilde O}(x,y) \rho_{r}^{z}(dx)\rho_{r}^{z}(dy) \leq \int_{\tilde O \cup B(z,r)} G_{\tilde O \cup B(z,r)}(x,y) \rho_{r}^{z}(dx)\rho_{r}^{z}(dy)
\end{equation}
where $G_{\tilde O \cup B(z,r)}$ denotes the Green function inside $\tilde O \cup B(z,r)$. Since $B(z,r) \subset \tilde O \cup B(z,r)$, we can explicitly compute the integral on the right hand-side of \eqref{niceCA}:
\begin{equation} \label{niceCA_2}
    \int_{\tilde O \cup B(z,r)} G_{\tilde O \cup B(z,r)}(x,y) \rho_{r}^{z}(dx)\rho_{r}^{z}(dy) = -\log r + \log \text{CR}(z, \partial(\tilde O \cup B(z,r))) \quad \text{almost surely.}
\end{equation}
Moreover, by the Koebe $1/4$-theorem, the definition of $J_{z,r}^{\cap}$ and the fact that $\Gamma^{(J_{z,r}^{\cap})}$ is almost surely path-connected, we almost surely have
\begin{equation*}
    \text{CR}(z, \partial( O(z) \cup B(z,r))) \leq 4\text{dist}(z, \partial (O(z) \cup B(z,r)) \leq 4r.
\end{equation*}
Combining \eqref{niceCA} and \eqref{niceCA_2} with this upper bound, we obtain,
\begin{equation*}
    \EE \bigg [ \bigg( \sum_{j: O_j \cap \partial B(z,r) \neq \emptyset} h_r^{(j)}(z) \bigg)^2 \, \big \vert \, \Gamma^{(J_{z,r}^{\cap})}\bigg] \leq \log \frac{4r}{r} = \log 4 
\end{equation*}
almost surely, which completes the proof of the lemma.
\end{proof}

\subsubsection{Proof that term (\ref{proba_nested}) is zero} To establish that this probability is equal to 0 (for any fixed $K>0$ and $c\ge 1/\sqrt{2}$), we are going to show that for any fixed $\tilde \delta > 0$,
\begin{equation*}
    \PP(\exists z \in \mathbb{D} 
    \text{ such that } \limsup_{r\to 0} \vert \tilde S_{r}(z) - \tilde S_{r_n(r)}(z_{n(r)}(z)) \vert > \tilde \delta ) = 0.
\end{equation*}
Recall that $r_n(r)=(1-c/n(r))n(r)^{-K}$ where $n(r)$ is the unique $n \in \mathbb{N}^{*}$ such that $r \in [n^{-K}, (n-1)^{-K})$). So let us fix $\tilde \delta > 0$ and further define
\begin{equation*}
    R_n(r) := \big(1+\frac{d}{n(r)}\big)n(r)^{-K}
\end{equation*}
where $d=d(c,K)>0$ is chosen large enough such that for all $z \in B(0,1/2)$ and $r>0$ 
\begin{equation*}
    \partial B(z,r) \subset B(z_{n(r)}(z),R_{n}(r)) \setminus B(z_{n(r)}(z),r_n(r)).
\end{equation*}
Note that since $\vert z - z_{n(r)}(z) \vert \leq n(r)^{-(K+1)}/\sqrt{2}$ and  $r \geq n(r)^{-(K+1)}/\sqrt{2} + r_n(r)$ (because $c\ge 1/\sqrt{2}$)
we do have that $\partial B(z,r)$ lies outside of $B(z_{n(r)}(z),r_n(r))$. Moreover, $\partial B(z,r)$ lies inside $B(z_{n(r)}(z), R_n(r))$ as long as 
$r \leq \vert z - z_{n(r)}(z) \vert + R_n(r)$, 
which is satisfied if
\begin{equation*}
    (n(r)-1)^{-K} \leq \frac{n(r)^{-(K+1)}}{\sqrt{2}} + (1+\frac{d}{n(r)})n(r)^{-K},
\end{equation*}
i.e. as long as $d$ is large enough. 

With our choice of $d$, if the event $\{ \vert \tilde S_{r}(z) - \tilde S_{r_n(r)}(z_{n(r)}(z)) \vert > \tilde \delta\}$ occurs for some $z \in B(0,1/2)$ and $r>0$, then there exist at least $C\log(n(r)-1)$ loops that surround $B(z_{n(r)}(z),r_n(r))$ but not $B(z_{n(r)}(z), R_n(r))$, for some constant $C=C(\tilde\delta, K)>0$. Indeed, we have:
\begin{align*}
    \vert \tilde S_{r}(z) - \tilde S_{r_n(r)}(z_{n(r)}(z)) \vert = &\big \vert \frac{S_r(z)}{\log 1/r} - \frac{S_{r_n(r)}(z_{n(r)}(z))}{\log 1/r_n(r)} \big \vert \\
    \leq & \frac{1}{K\log(n(r)-1)} \vert S_r(z) - S_{r_n(r)}(z_{n(r)}(z)) \vert\\
    \leq &\frac{2\lambda}{K\log(n(r)-1)} \\
    & \times \#\{ \ell \in \Gamma \text{surrounding $B(z_{n(r)}(z),r_n(r))$ but not $B(z_{n(r)}(z),R_n(r))$} \}.
\end{align*}
where the last inequality follows because the number of loops surrounding $B(z_{n(r)}(z),r_n(r))$ but not $B(z,r)$, i.e. the loops that contribute to $\vert S_r(z) - S_{r_n(r)}(z_{n(r)}(z)) \vert$, is less than the number of loops surrounding $B(z_{n(r)}(z),r_n(r))$ but not $B(z_{n(r)}(z),R_n(r))$, 
and because the signed Bernoulli random variable associated to each loop almost surely has modulus $\le 2\lambda$. See Figure \ref{fig_Bn} for a visual representation.
\begin{figure}
\begin{center}
      \includegraphics[width=0.5\textwidth,trim={4cm 0 0 0},clip]{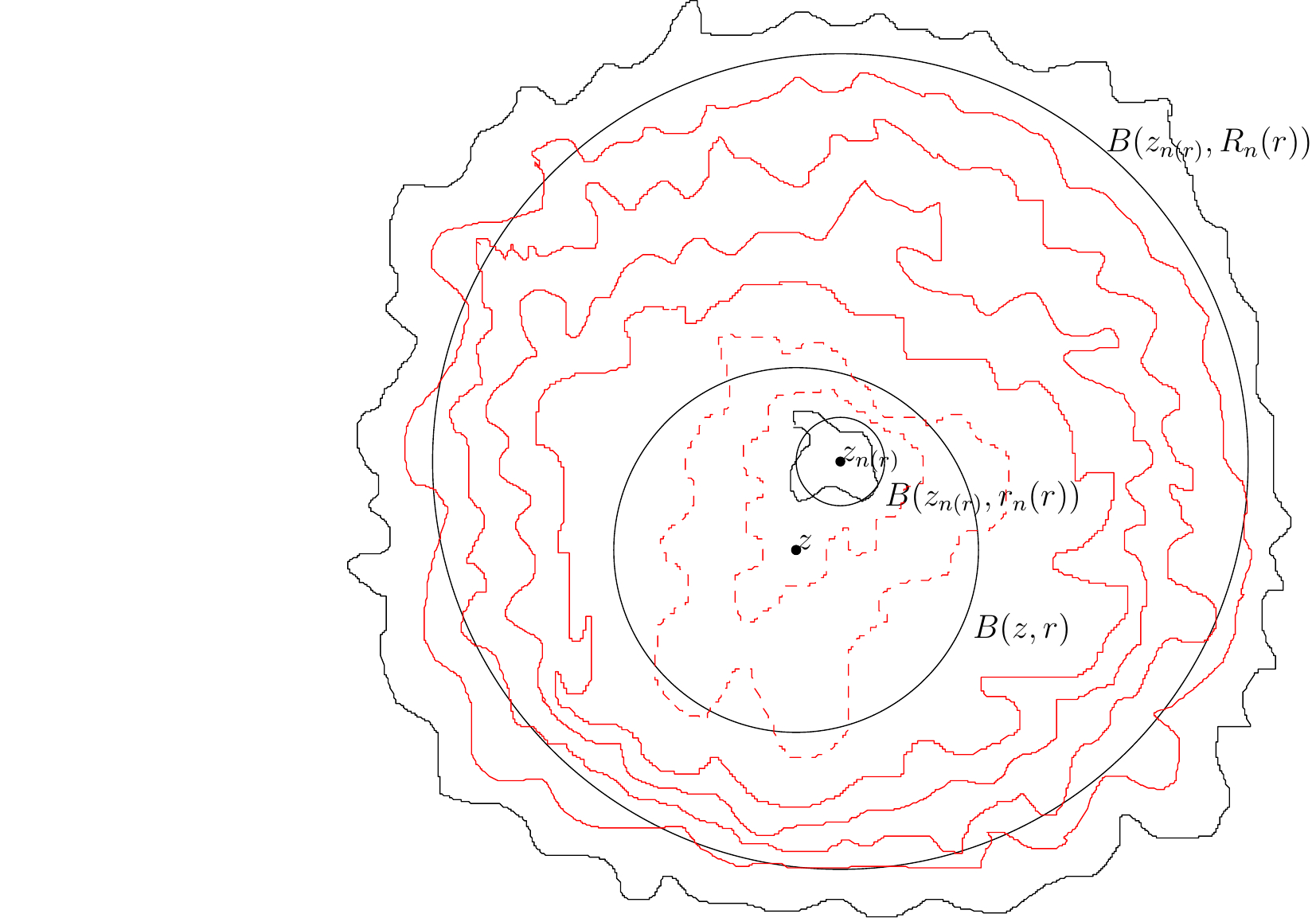}
  \caption{The balls $B(z,r)$, $B(z_{n(r)},r_n(r))$ and $B(z_{n(r)},R_n(r))$. The red loops that are dashed contribute to $\vert S_r(z) - S_{r_n(r)}(z_{n(r)}(z)) \vert$. The red loops, either plain or dashed, are those $\mathrm{CLE}_4$ loops that surround $B(z_{n(r)}, r_n(r))$ but not $B(z_{n(r)}, R_n(r))$.}
  \label{fig_Bn}
  \end{center}
\end{figure}

For $n \geq 2$ and $z_n \in D_n$, let us define the event 
\begin{equation*}
    A_{n, z_n} = \{ \text{there are at least $C\log(n-1)$ loops that surround $B(z_n,r_n)$ but not $B(z_n,R_n)$} \}
\end{equation*}
where $r_n=(1-c/n)n^{-K}$, $R_n=(1+d/n)n^{-K}$ and $C$ is the constant derived above. From the previous discussion, we obtain the following inequality:
\begin{equation*}
    \PP(\exists z \in \mathbb{D} 
    \text{ such that } {\limsup_{r\to 0}} \vert \tilde S_{r}(z) - \tilde S_{r_n(r)}(z_{n(r)}(z)) \vert > \tilde \delta ) \leq \PP\bigg( \big \{ \underset{z_n \in D_n}{\bigcup} A_{n,z_n}\big \} \text{ i.o.}\bigg).
\end{equation*}
So we need to show that the right-hand side in this inequality is equal to 0. This will follow from Borel-Cantelli lemma if we can establish that the sum
\begin{equation*}
    \sum_{n \geq 2} n^{2(K+1)} \max_{z_n \in D_n} \PP(A_{n, z_n})
\end{equation*}
is finite. This holds in particular if $\max_{z_n \in D_n} \PP(A_{n, z_n})$ decays faster than any power of $n$. In other words, it suffices to prove the following lemma.

\begin{lemma}\label{lem_A_bound}
There exists a sequence $g(n)\to 0$ as $n\to \infty$ such that \begin{equation*}
    \PP(A_{n,z_n}) \leq  g(n)^{\lfloor C\log(n-1) \rfloor -1}
\end{equation*}
for all $n\ge 2$ and $z_n\in D_n$.
\end{lemma}

\begin{proof}
Let us fix $n\ge 2$ and $z_n\in D_n$.
To lighten the notations, let us set 
{$J= J_{z_n,R_n}^{\cap}$ so that $\ell^J$ is the first nested CLE$_4$ loop intersecting $B(z_n,R_n)$.} Using the the nestedness of the CLE$_4$ loops, we can bound 
\begin{align} \label{ineq_An}
    \PP(A_{n,z_n}) &\leq \PP \big(B(z_n,r_n) \subset \ell^{J + \lfloor C \log(n-1) \rfloor -1} \big) \nonumber \\
    &
\le   \PP \big( B(z_n,r_n) \subset \ell^{J + \lfloor C \log(n-1) \rfloor -1} \, \vert \, B(z_n,r_n) \subset \ell^J \big) \nonumber \\
    &\le  \prod_{k=1}^{\lfloor C \log(n-1) \rfloor -1} \PP \big( B(z_n,r_n) \subset \ell^{J+k} \, \vert \,  B(z_n,r_n) \subset \ell^{J+k-1}  \big),
\end{align}
where we also used the trivial bound $\mathbb{P}(B(z_n,r_n)\subset \ell^{J})\le 1$ in the second line.

We will show that there exists $g(n)$ not depending on our choice of $z_n\in D_n$, and with $g(n)\to 0$ as $n\to \infty$, such that for every $k \in \{1, \dots, \lfloor C \log(n-1) \rfloor \}$,
\begin{equation} \label{k_pn2}
    \PP(B(z_n,r_n) \subset \ell^{J+k}  \vert B(z_n,r_n)\subset \ell^{J+k-1}) \leq g(n).
\end{equation}
This clearly implies the lemma, by \eqref{ineq_An}. 

To see \eqref{k_pn2}, observe that conditionally on $\ell^{J+k-1}$ (for any $k\ge 1$), $\ell^{J+k}$ has the law of the (unique) CLE$_4$ loop surrounding $z_n$ in a \emph{non-nested} CLE$_4$ in $\ell^{J+k-1}$. Moreover, $\ell^{J+k-1}\subset \mathbb{D}$ contains a point within distance $R_n$ of $z_n$ by definition. Thus, by \cite[Theorem 4-6]{Ahlfors}, the extremal distance between $\ell^{J+k-1}$ and $\partial B(z_n,r_n)$ is deterministically bounded above by $e(n)$: the extremal distance between the unit circle and the line segment $[R_n/r_n,+\infty]$. Notice that $R_n/r_n\to 1$ by construction as $n\to \infty$, and therefore by continuity of extremal distance (see for example \cite{contED}), we have that $e(n)\to 0$ as $n\to \infty$. {By \cite[Theorem 4-1]{Ahlfors}, it follows that on the event $B(z_n,r_n) \subset \ell^{J+k}$, also the extremal distance between $ \ell^{J+k}$ and $ \ell^{J+k-1}$ is bounded by $e(n)$. But this probability tends to $0$ by \cite[Theorem 1.1]{EDCLE} and conformal invariance of CLE$_4$.}
\end{proof}

\subsection{Proof of Proposition \ref{prop_singular} and Corollary \ref{cor_BTLS}}

Using the results of the previous subsection, we now turn to the proof of Proposition \ref{prop_singular}.

\begin{proof}[Proof of Proposition \ref{prop_singular}]
Let $h$ be a Dirichlet GFF in $\D$. Let $\Gamma$ be a nested CLE$_4$ coupled to $h$ as described in the discussion before Theorem \ref{th_CLE} and let $z \mapsto S_r(z)$ be the corresponding weighted $\mathrm{CLE}_4$ nesting field. Observe that
\begin{equation*}
    \{ z \in \D : \limsup_{r \to 0} \frac{\sqrt{2\pi} S_r(z)}{\log 1/r} > Q(\xi), z \in \bigcup_{n \in \mathbb{N}^{*}} \Gamma^{(n)}\} = \emptyset \quad \text{{almost surely}.}
\end{equation*}
Therefore, the same arguments as in the proof of Corollary \ref{th_nestingfield} show that the set
\begin{equation*}
    \{z \in \D : \limsup_{r \to 0} \frac{S_r(z)}{\log 1/r} > Q(\xi) \}
\end{equation*}
is {almost surely} totally disconnected. Theorem \ref{th_CLEGFF} then allows us to conclude that $S_{h}^{\xi}(\D)$ is {almost surely} totally disconnected.
\end{proof}

From Theorem \ref{th_CLEGFF} and Corollary \ref{th_nestingfield}, we finally deduce Corollary \ref{cor_BTLS}.

\begin{proof}[Proof of Corollary \ref{cor_BTLS}]
Choose $M$ such that $K \leq 2\lambda(M-1)$. Then, by \cite[Proposition~3]{BTLS}, $A \subset \TVS_{-2M\lambda, 2M\lambda}$ {almost surely} where $\TVS_{-2M\lambda, 2M\lambda}$ denotes the two-valued set of level $-2M\lambda$ and $2M\lambda$ of $h$. As explained in \cite[Section~1.2]{BTLS}, $\TVS_{-2M\lambda, 2M\lambda}$ can be constructed from a nested CLE$_4$ $\Gamma$ in $\mathbb{D}$ coupled to $h$. Let us briefly recall this construction, which is the key to show the corollary. For $z \in \mathbb{D}$, recall that $\ell_{z}^{j}$ denotes the loop of $\text{Loop}^{(j)}(\Gamma)$ containing $z$. Then, for any $n \geq 1$, in the local set coupling $(h, \Gamma^{(n)})$, the value of the harmonic function in $\text{int} (\ell_{z}^{n})$ is given by
\begin{equation*}
    H_n(z) = \sum_{j=1}^{n} \xi_{\ell_{z}^{j}}
\end{equation*}
where $\PP(\xi_{\ell_{z}^{j}} = 2\lambda) = \PP(\xi_{\ell_{z}^{j}} = -2\lambda) =1/2$ and $(\xi_{\ell_{z}^{j}})_{1\leq j\leq n}$ are independent random variables. Moreover, if $z, z' \in \mathbb{D}$ are such that $\ell_{z}^{n} = \ell_{z'}^{n}$, then $H_n(z)=H_n(z')$. To construct $\TVS_{-2M\lambda, 2M\lambda}$ from $\Gamma$, for each $z \in \mathbb{Q}^{2} \cap \mathbb{D}$, we define $\tau_M(z) := \inf \{n \geq 1: \vert H_n(z) \vert = 2M\lambda \}$. $\tau_M(z)$ is {almost surely} finite since $(H_n(z))_{n \geq 1}$ is a simple random walk and $O^{\tau_M(z)}(z) := \text{int} (\ell_{z}^{\tau_{M}(z)})$ is then almost surely an open and simply connected set. Set 
\begin{equation*}
    A_M := \mathbb{D} \setminus \bigcup_{z \in \mathbb{Q}^{2} \cap \mathbb{D}} O^{\tau_M(z)}(z).
\end{equation*}
$A_M$ is a local set coupled to $h$ such that the corresponding harmonic function {almost surely} takes values in $\{-2M\lambda, 2M\lambda\}$ when restricted to a connected component of $\mathbb{D}\setminus A_M$. Therefore, by \cite[Proposition~1]{BTLS}, $A_M= \TVS_{-2M\lambda, 2M\lambda}$ {almost surely}.

This construction of $\TVS_{-2M\lambda, 2M\lambda}$ shows in particular that $\TVS_{-2M\lambda, 2M\lambda} \subset \cup_{n \geq 1} \Gamma^{(n)}$ {almost surely}. But by Theorem \ref{th_CLEGFF}, for any $\gamma \in (0,2]$, {almost surely}
\begin{align*}
    \big \{ z \in \mathbb{D}:  \lim_{r \to 0} \frac{h_r(z)}{\log 1/r} = \frac{\gamma}{\sqrt{2\pi}}, z \in \bigcup_{n \geq 1} \Gamma^{(n)} \big \} = \big \{ z \in \mathbb{D}:  \lim_{r \to 0} \frac{S_r(z)}{\log 1/r} = \frac{\gamma}{\sqrt{2\pi}}, z \in \bigcup_{n \geq 1} \Gamma^{(n)} \big \} = \emptyset.
\end{align*}
This completes the proof of Corollary \ref{cor_BTLS}, since the previous discussion shows that $A \subset \cup_{n \geq 1} \Gamma^{(n)}$ {almost surely}.
\end{proof}

Let $h$ be a GFF in $\mathbb{D}$ with Dirichlet boundary conditions. A direct calculation shows that if $z \in \partial \mathbb{D}$ and $r>0$, then $\EE[h_r(z)^2]$ is bounded by a constant independent of $z$ and $r$. By adapting the proof of Lemma 3.1 in \cite[Lemma~3.1]{thick_points}, we can deduce the following slight refinement of Corollary \ref{cor_BTLS}: if $A$ is a $K$-BTLS coupled to $h$, $K > 0$, then, for any $\gamma \in (0,2]$,
\begin{equation*}
    \big \{ z \in A \cup \partial \mathbb{D}:  \lim_{r \to 0} \frac{h_r(z)}{\log 1/r} = \frac{\gamma}{\sqrt{2\pi}} \big \} = \emptyset \quad \text{{almost surely.}}
\end{equation*}

\newpage
\bibliographystyle{alpha}
\bibliography{biblio}

\end{document}